\documentclass[a4paper]{amsart}
\usepackage[centering]{geometry}
\usepackage{lmodern}
\usepackage[T1]{fontenc}
\usepackage[utf8]{inputenc}
\usepackage{mathrsfs}
\usepackage{amssymb,amsthm}
\usepackage{tensor}
\usepackage{slashed}
\usepackage{wrapfig}
\usepackage[pdftex]{graphicx}
\usepackage{esint}
\usepackage{epstopdf}
\usepackage{amscd}
\usepackage{enumerate}
\usepackage{fancyhdr} 
\usepackage{microtype}
\usepackage[all]{xy}
\xyoption{poly} 
\usepackage{hyperref}
\usepackage{easyReview}
\usepackage{tikz-cd}
\tikzcdset{arrow style=tikz, diagrams={>=stealth}}
\usepackage[capitalise]{cleveref}
\makeatletter
\newcommand{\crefnames}[3]{%
\@for\next:=#1\do{%
\expandafter\crefname\expandafter{\next}{#2}{#3}%
}%
}
\makeatother
\crefnames{part,chapter,section}{\textsection}{\textsection\textsection}
\crefname{lem}{Lemma}{Lemmas}
\Crefname{lem}{Lemma}{Lemmas}
\crefname{thm}{Theorem}{Theorems}
\Crefname{thm}{Theorem}{Theorems}
\crefname{prop}{Proposition}{Propositions}
\Crefname{prop}{Proposition}{Propositions}
\crefname{defn}{Definition}{Definitions}
\Crefname{defn}{Definition}{Definitions}
\theoremstyle{plain}%
\newtheorem{thm}{Theorem}[section]
\newtheorem{prop}[thm]{Proposition}
\newtheorem{lem}[thm]{Lemma}
\newtheorem{cor}[thm]{Corollary}

\theoremstyle{definition}
\newtheorem{defn}{Definition}[section]

\theoremstyle{remark}

\newtheorem*{rem}{Remark}

\numberwithin{equation}{section}

\DeclareMathOperator{\grad}{grad}

\DeclareSymbolFont{AMSb}{U}{msb}{m}{n}
\DeclareMathSymbol{\N}{\mathbin}{AMSb}{"4E}
\DeclareMathSymbol{\Z}{\mathbin}{AMSb}{"5A}
\DeclareMathSymbol{\R}{\mathbin}{AMSb}{"52}
\DeclareMathSymbol{\Q}{\mathbin}{AMSb}{"51}
\DeclareMathSymbol{\I}{\mathbin}{AMSb}{"49}
\DeclareMathSymbol{\F}{\mathbin}{AMSb}{"46}

\newcommand{\set}[1]{\left\{#1\right\}}

\newcommand{\brac}[1]{\ensuremath{\left( {#1} \right)}}
\newcommand{\sbrac}[1]{\ensuremath{\left[ {#1} \right]}}%

\newcommand{\ddif}{\blacktriangle} 
\newcommand{\decop}{\underline{d}}
\newcommand{\incop}{\underline{s}}
\newcommand{\cycop}{\underline{t}}

\newcommand{\degenmaps}{\pmb\sigma}
\newcommand{\facemaps}{\pmb\delta}
\newcommand{\cyc}{\pmb\tau}
\newcommand{\funsp}{\mathcal{C}_{\mathrm{Spec}}}

\usepackage{mathtools}
\newcommand{\defeq}{\vcentcolon=}
\newcommand{\hotimes}{\mathbin{\hat\otimes}}
\author{Yang Liu}
\address{
SISSA,
via Bonomea 265,
34136 Trieste,
Italy
}

\email{yliu@sissa.it}
\title {
    Cyclic Structure behind Modular Gaussian Curvature
}

\keywords{
    Modular Curvature, Variational Calculus, Cyclic Category  
}

\subjclass[2010]{47A60, 46L87,  58Exx}

\date{\today} 

\begin{document}

\begin{abstract}
We propose a systematic scheme for computing the  variation of rearrangement
operators arising in the recently developed spectral geometry on noncommutative
tori and $\theta$-deformed Riemannian manifolds. It can be summarized as
a category whose objects consists of spectral functions of the rearrangement
operators and morphisms are generated by transformations associated to basic
operations of the variational calculus. The generators of the morphisms fulfil
most of the relations in Connes's cyclic category, but also include all the
partial derivatives. Comparison with Hopf cyclic theory has also been made.
\end{abstract}

\maketitle

\tableofcontents 

\section{Introduction}
\label{sec:intro}

This paper is devoted to the variational aspects of spectral geometry on
noncommutative tori initiated in \cite{Connes:2011tk} and later
\cite{MR3148618,MR3194491}.
The new features that we would like to investigate concerns purely the
noncommutativity  between the metric coordinates and their derivatives.
The very first example in this direction is a simple functional relation
derived from the variational nature of the modular Gaussian curvature
introduced in \cite{MR3194491}.
As demonstrated  in the $a_4$-term calculation \cite{2016arXiv161109815C},
the complexity of such relations from variation grow dramatically when
the order of differentiations involved increases
\footnote{Modular Gaussian
curvature is determined by the $a_2$-term, here $2$ and $4$ agree with the
maximal order of derivatives appeared in the calculation.}.
One of our primary motivations  is to search for new cancellation and
simplification for the lengthy formulas obtained  \cite{2016arXiv161109815C},
especially those in the appendix.
Such goal  often relies on adding new mathematical structures, 
the paper contributes  some partial progress  in relation to the cyclic
(co)homology , discovered independently by Connes and Tsygan.

For the $a_2$-term, the cancellation can be seen from the improvement between
\cite{LIU2017138} and \cite{Liu:2018ab}. In the former paper   \cite{LIU2017138},
in which the higher dimensional analogue of the Gaussian curvature was first
introduced on toric noncommutative manifolds
\footnote{a.k.a $\theta$-deformations \cite{Rieffel:1993wb}
or Connes-Landi deformations, cf.  \cite{Connes:2001wv}, \cite{Connes:2002wh},
\cite{Brain:2013wp}.},
the functional relations take several equations
\cite[Eq. (4.5) - (4.8)]{LIU2017138}, which were
achieved on noncommutative four tori earlier   in
\cite{MR3369894}\footnote{at the end of the proof of Thm 5.1.}.
After dwelling on the computation for a while, the author realized that
those less organized relations indeed admits a much simpler form
\cite[Thm. 2.15]{Liu:2018ab} akin to  \cite[Eq. (4.42)]{MR3194491} 
\footnote{details are also provided in \cref{sec:CM-relt}.}.
Later, the author had a short visit at IHES and shared the progress with Connes.
He pointed out immediately that the relations in \cite[Prop 2.11]{Liu:2018ab},
which explain the aforementioned simplification, 
resemble exactly the compatibility condition
(cf. \eqref{eq:simp+cyc-rels-V-2-1}) in the cyclic category $\Delta C$.
The main results, \cref{thm:cyc-cato-sim-I} and \cref{cor:op-cyc-cate-relts},
constitute a more comprehensive interpretation of his vision.

The paper is organized as follows. In \cref{sec:rearr-op-fun}, we recall
the construction of rearrangement operators and then proceed with the
variational calculus \cref{sec:vcal-reops}.
In the literature, rearrangement operators are defined in terms of the
derivation commutator $\mathbf x = [ \cdot , h]$ associate to a log-conformal 
factor $h$ due to its own geometric significance \cite[\S~1.5]{MR3194491}.
In comparison to \cite[\S~2]{Liu:2018ab}, we now choose to work with the
multiplication operators associated to $h$ that makes the simplicial structure
in the variational calculus more transparent.
The difference occurs  only at the level of notations, that is, one can
pass from one to the other by the change of variable
given in  \eqref{eq:x-to-h}.
In \cref{sec:main}, we encapsulate the technical discussions in the previous
two sections into a category $ \mathcal C$, which possesses most of the axioms
as a cyclic module. The core is the generating set
of morphisms and relations stated in the two theorems. 
To further conceptualize the computations in \cref{sec:vcal-reops},
we make some comparison with Hopf cyclic theory in \cref{sec:cmp-Hcyc}.
In that regard, the failure of $ \mathcal C$ being a cyclic module can be explained as
the incompatibility between the underlying algebra and coalgebra structures.
Lastly, we go back to spectral geometry for applications in \cref{sec:appl-modgeo}.
We first reproduce, in \cref{sec:CM-relt}, Connes and Moscovici's functional
relation obtained in \cite{MR3194491} to illustrate how morphisms of $\mathcal C$
govern the variational calculus. 
In particular, it prodivdes more clear interpretation for the
three terms in  \eqref{eq:CMtp-II}  in cyclic theory. 
Moreover, the method also has the merit of being  suitale for computer algebra
system implementation, 
which seems to be indispensable for simplifying  the $a_4$-term
calculation \cite{2016arXiv161109815C}.
The second application concerns the action of $ \mathcal C$ on functionals 
arising from rearrangement lemma,
which could be useful to the verification of the functional relations,  a more
formidable task (but equally important) than deriving the relations. 
In \cref{sec:fn-psdo}, we only consider much simpler family of functions 
$\set{\omega_\alpha}$ (but intricately connected to those hypergeometric integrals
introduced in \cite{Liu:2018aa,Liu2020General-Rearran}),
and show that the action of $ \mathcal C $ can be used to derived some
differential relations. 

The notion of rearrangement operators is closely related to multiple operator
integral (see the end of \cref{sec:rearr-op-fun} for references), which now plays
a crucial role in study of spectral actions \cite{Nuland:2021tw}. 
We hope to find new applications of $ \mathcal C$ in that direction in furture
publications.

\section{Rearrangement Operators via Schwartz Functional Calculus}
\label{sec:rearr-op-fun}

Our discussion in \cref{sec:rearr-op-fun,sec:vcal-reops} is greatly influenced  by 
the work of Lesch  \cite{leschdivideddifference}, 
we shall follow the notations there,  with slightly changes,
to construction the rearrangement operators. 
Let $ \mathscr A $ be a unital  $C^*$-algebra. For any $h \in \mathscr A$,
denote by $h^{(j, n)} \in  L(\mathscr A^{\otimes n} , \mathscr A) $
the linear operator from $ \mathscr A^{\otimes  n} \to \mathscr A$:
\begin{align}
\label{eq:h-j-n-mlt-jslot}
h^{(j , n)} : \mathscr A^{\otimes  n} \to \mathscr A:
\rho_1 \otimes  \cdots \otimes \rho_n \mapsto
\rho_1 \cdots  \rho_{ j-1} h \rho_j \cdots \rho_n ,
\end{align}
that is the multiplication at the $j$-th slot, $j=0, \ldots, n$.
For $n=1$,  $ h^{(0 , 1)} , h^{(1 , 1)} : \mathscr A \to \mathscr A$ are
simply the left and the right multiplications.

Now assume $h = h^* \in  \mathscr A$ is self-adjoint and $U \subset \R$
is a open neighborhood that contains the spectrum of $h$.
For any $f \in C^\infty_c(U^{n+1})$, let $ \widehat{ f}$  be the normalized
Fourier transform such that
\begin{align}
\label{eq:hat-f-nrm}
f(x_0, \ldots , x_n) =
\int_{ \R^{n+1} }  \widehat{ f} (\xi)
e^{i x \cdot \xi} d\xi,
\end{align}
where $x= (x_0, \ldots , x_n)$ and $\xi = (\xi_0 , \ldots , \xi_n)$.

\begin{defn} [Schwartz functional calculus]
Let us fix a self-adjoint $h \in  \mathscr A$.
There exists an linear map
\begin{align*}
\mathscr S_h^{(n)}: C^\infty_c(U^{n+1}) \to
L( \mathscr A^{\otimes  n+1}, \mathscr A),
\end{align*}
given by
\begin{align}
\label{eq:scr-S-f-defn}
\mathscr S_h^{(n)} (f) = f (h^{(0)}, \ldots , h^{(n)}) \defeq
\int_{ \R^{n+1} }  \widehat{ f} (\xi)
e^{i \xi_0 h^{(0)}} \cdots e^{i \xi_n h^{(n)}}
d\xi,
\end{align}
where the right hand side,
as an operator from $ \mathscr A^{\otimes  n} \to \mathscr A$,
means (cf. \eqref{eq:h-j-n-mlt-jslot}):
\begin{align}
\label{eq:scr-S-f-of-rho}
\mathscr S_h^{(n)} (f)  ( \rho) =
f (h^{(0)}, \ldots , h^{(n)}) ( \rho)
=
\int_{ \R^{n+1} }  \widehat{ f} (\xi)
e^{i \xi_0 h} \rho_1 e^{i \xi_1 h} \cdots
\rho_n e^{i \xi_n h }
d\xi ,
\end{align}
where
$ \rho = \rho_1 \otimes  \cdots  \otimes  \rho_n  \in  \mathscr A^{\otimes  n}$.
\end{defn}
Let us make some remarks on nations.
In later exploration, the short form $  \mathscr S_h^{(n)} (f)$
often take place when the arguments $ \brac{ h^{(0)}, \ldots , h^{(n)} }$ are 
fixed. Meanwhile, in most of the calculations, we have to manipulate 
$ \brac{ h^{(0)}, \ldots , h^{(n)} }$, for instance,
performing cyclic permutation on them, 
the full form $ f \brac{ h^{(0)}, \ldots , h^{(n)} }$ is inevitable
in order to keep track of the change of variables.
We will freely drop the superscript  of $  \mathscr S_h^{(n)}$ whenever the domain
$ \mathscr A^{\otimes  n}$ is clear from context, or we can interpret
$ \mathscr S_h = \bigoplus_{ n=1}^\infty \mathscr S_h^{(n)}$ as
the direct sum.

Denote by $\funsp(n) \defeq \funsp(n,h) $ the collection of all $(n+1)$-variable
functions $f ( x_0 , \ldots , x_n)$
such that the Schwartz functional calculus
$ \mathscr S_h^{(n)} (f) \in  L( \mathscr A^{\otimes  n+1}, \mathscr A) $
is well-defined. To be more specific, we can consider a subalgebra
$ C_c^\infty(U_h^{n+1}) \subset \funsp(n,h)$ of smooth function with compact
support in  the open set
$ U_h^{n+1} = U_h \times \cdots \times  U_h \subset \R^{n+1}$,
and $U_h \subset \R$ is some open subset containing the spectrum of
$h  \in \mathscr A$.
Such smooth spectral functions are sufficient for many applications, especially
for the heat coefficients computation on  noncommutative tori.
Thanks to the nuclear property of the smooth Fr\'echet topology,
there exists a unique completion  $ \hotimes $ of the algebraic tensor $\otimes $
so that $ C_c^\infty(U_h \times \cdots \times U_h )
\cong C_c^\infty(U_h)^{ \hotimes (n+1) }$.
In more elementary terms, the algebraic tensors in
$ C_c^\infty(U_h)^{\otimes (n+1)}$, which is a dense subset,
correspond to function of separating variables on $U_h^{ n+1 }$:
\begin{align}
\label{eq:ele-tns-2-fun-sp-var}
f_0 \otimes \cdots \otimes  f_n \mapsto
\brac{ (x_0 , \ldots , x_n ) \mapsto
f_0 (x_0) \cdots f_n (x_n) \in  C_c^\infty(U^{n + 1}) } .
\end{align}
We will come back to this class of functions in \cref{sec:cmp-Hcyc}.

In the context of conformal geometry on noncommutative manifolds mentioned at
the introduction, the self-adjoint $h$ represents a log-conformal factor while
the  $\rho_i$'s are the derivatives of $h$. 
The notation   $f (h^{(0)}, \ldots , h^{(n)}) ( \rho)$
in \eqref{eq:scr-S-f-of-rho} implements a rearrangement that moves all the
zero-order derivatives to the left.
Such rearrangement process can be indeed constructed for a much larger class of
functions than the Schwartz ones.  The theory of multiple operator integrals 
\cite{Potapov:2013uv,Skripka:2019up,Dykema:2009wm,Pagter:2002vy}
explores the generality of the function $f$ and the operator $h$ for which the
integral in \eqref{eq:scr-S-f-of-rho} is well-defined. 
Connection with cyclic theory has also been made \cite{Nuland:2021tw} 
in the study of spectral actions.

\section{Variational Calculus  for  Rearrangement Operators}
\label{sec:vcal-reops}

This section is a continuation of the systematic approach for the standard
differential calculus with the appearance of rearrangement operators initiated in
\cite{Liu:2018ab}, where only spectral functions of one and two variable are
considered. 
We start with recalling basis notations on divided difference, whose crucial role 
in the our context of variational calculus was first pointed out in
\cite{leschdivideddifference}. 
\subsection{Divided Differences}
\label{sec:ddiff} 
For a one-variable function $f(z)$, the $n$-divided difference
$ \ddif^n (f)  $ has $(n+1)$-arguments which is
inductively  defined  as follows:
\begin{align*}
\ddif^0 (f) (x_0) & \defeq f (x_0) \\
\ddif^n(f) \brac{ x_0 , \ldots , x_n }
& \defeq
\left(
\ddif^{n-1} (f) \brac{ x_0, \dots, x_{n-1} } -
\ddif^{n-1} (f) \brac{ x_1, \dots, x_{n} }
\right)/ (x_0 - x_n) .
\end{align*}
The first divided difference will be frequently used later:
\begin{align*}
    \ddif (f)( x_0,x_1 ) = (f(x_0) - f(x_1)) / (x_0 - x_1).
\end{align*}
One can prove, by induction,  the explicit formula: 
\begin{align}
\ddif^n (f) (x_0, x_1, \dots, x_n)
= \sum_{l=0}^n
f(x_l)  \prod_{s=0, s \neq l}^n (x_l -x_s)^{-1}.
\label{eq:divided-gen-formula}
\end{align}
It is also quite convenient to use squared brackets  to enclose new variables
generated by iterated divided differences and then skip $\ddif^n$:
\begin{align}
\label{eq:f-sbracket} 
    f [ x_0 , \ldots , x_n] \defeq   \ddif^n (f) (x_0, x_1, \dots, x_n).
\end{align}
The following basic properties will be frequently used:
\begin{itemize}
\item Leibniz rule:
\begin{align}
\label{eq:Leibniz-rule}
(fg)[x_0, \dots,x_n] = f(x_0) g[x_0,\dots,x_n] 
+ f[x_0,\dots,x_n] g(x_n) .
\end{align}
\item Composition rule:
\begin{align}
(f[y_0,\dots,y_q, z])[x_0, \dots,x_p]_z = f[y_0,\dots,y_q, x_1, \dots, x_p] .
\label{eq:Composition-rule}
\end{align}
\item The confluent case: suppose there are $\alpha+1$ copies of $x$
in the arguments of the divided difference, then:
\begin{align}
\label{eq:ddif-confluet}
f[y,x, \ldots , x] = \frac{1}{\alpha !}
\partial_x^\alpha f[y,x] .
\end{align}

\item $f[x_0, x_1, \dots, x_n]$ is symmetric in all the arguments.
\end{itemize}

When multivariable functions are concerned, there are several intuitive 
notations to indicate the variable on which the divided difference acts,
for instance:
\begin{enumerate}[i)]
    \item 
using a subscript $ f (x_0 , \ldots , z_i, \ldots , x_n) [x_i, x_{ i+1} ]_{ z_i}$,
see also \eqref{eq:Composition-rule};
\item  using a placeholder $\bullet$:  
$ f (x_0 , \ldots , \bullet, \ldots , x_n) [x_i, x_{ i+1} ]$;

\item directly enclosing the variables
$ f (x_0 , \ldots , [x_i, x_{ i+1} ] , \ldots , x_n) $
    when there is no confusion on the position. 
\end{enumerate}

Much like partial derivatives, given a multiindex
$\alpha = (\alpha_0 , \ldots , \alpha_n) \in \Z^{n+1}_{ \ge 0}$ and
a function $f (x_0 , \ldots , x_n )$, one can form
$ \ddif^{\alpha} (f)$, which is the function obtained by
applying the divided difference $\alpha_j$-times 
on $j$-th argument of $f$\footnote{ the order does not matter
}, where $ 0\le  j \le  n$.
In accordance with \eqref{eq:f-sbracket},
evaluation of   $ \ddif^{ (\alpha_0 , \ldots , \alpha_n)} (f)$ reads
\begin{align}
\label{eq:divdif-alpha}
f \brac{  [ x_{ 0,0} , \ldots , x_{ 0, \alpha_0}] , \ldots ,
[ x_{ n,0} , \ldots , x_{ n, \alpha_n}]
} .
\end{align}
For instance, we can consider the divided difference of $f$ at the $j$-th argument:
$\ddif_j (f) \defeq \ddif^{(0, \ldots , 1 , \ldots ,0)} (f)$, with
$j=0, \ldots , n$, which is part of the face maps defined   
$ \facemaps^{(j, n)}$ in \cref{eq:face-j-for-h}:
\begin{align*}
\facemaps^{(j , n)} (f) (x_0  , \ldots , x_{ n+1} ) \defeq
f \brac{ x_0 , \ldots ,[ x_j , x_{ j+1 } ] , x_{ j +2 },  \ldots , x_{ n + 1 } }.  
\end{align*}
Note that, beside $ \ddif_j (f)$,  there is a shift of position of the
arguments after index $j$:
that is $x_{ j+2}$ is at the $(j+1)$-th slot, 
$ \ldots$, $x_{ n+1}$ is at the $n$-th slot.
In fact, such shift can be implemented by the $j$-th face map
$\delta_j^{(n)}: [n] \to [n+1]$ in the simplicial category  $\Delta$, 
cf. \cref{sec:cyc-DelC}:
\begin{align*}
\delta_j^{(n)}: 0\mapsto 0, \ldots , j-1 \mapsto  j-1, \, \,
j\mapsto j+1, \ldots , n \mapsto  n+1.
\end{align*}
One expects the corresponding simplicial relations:
\begin{prop}
\label{prop:spl-rlt-del-i-j-prf}
The operators $ \facemaps^{(j ,n)} : \funsp(n) \to \funsp (n+1)$, 
with $ j=0, \ldots , n$, satisfy 
\begin{align}
\label{eq:spl-rlt-del-i-j-prf}
\facemaps^{(j , n+1)} \facemaps^{(i , n)} =
\facemaps^{(i,n+1)} \facemaps_{  } ^{( j-1, n)} , \, \,
\text{ for}  \,
i < j .
\end{align}
\end{prop}
\begin{proof}
Start with the left hand side $ \facemaps^{( i , n)} (f)$:
\begin{align*}
f (x_0 , \ldots , x_{ n} )
\xrightarrow{ \facemaps^{(i , n)}}
f  \brac{ x_0 , \ldots ,[ x_i , x_{ i+1 } ] , \ldots , x_{ n + 1 }   } ,
\end{align*}
to further apply $ \facemaps^{( j , n+1)}$, since $j >i$, the divided difference
occurs on $\set{x_{ i+1} , \ldots , x_{ n+1}  }$, thus
\begin{align*}
&\, \,
f  \brac{ x_0 , \ldots ,[ x_i , x_{ i+1 } ] , \ldots , x_{ n + 1 }   } \\
\xrightarrow{ \facemaps_j^{(n + 1)}} &\, \,
\begin{cases}
f  \brac{ x_0 , \ldots ,[ x_i , x_{ i+1 } , x_{ i+2} ] , \ldots , x_{ n + 1 }   }
& \text{for $j = i+1$} \\
f  \brac{ x_0 , \ldots ,[ x_i , x_{ i+1 } ] , \ldots, [ x_j , x_{ j+1 } ]
\ldots , x_{ n + 1 }   }
& \text{for $j > i+1$ .}
\end{cases}
\end{align*}
On the other side,
\begin{align*}
f (x_0 , \ldots , x_{ n} )
\xrightarrow{ \facemaps ^{( j-1, n)}}
f  \brac{ x_0 , \ldots ,[ x_{ j-1}  , x_{ j } ] , \ldots , x_{ n + 1 }   } .
\end{align*}
In the next step, the divided difference is taken at the $i$-th argument,
which is before $j-1$: $0 \le i \le  j-1$,
\begin{align*}
&\, \,
f  \brac{ x_0 , \ldots ,[ x_{ j-1}  , x_{j} ] , \ldots , x_{ n + 1 }   } \\
\xrightarrow{ \facemaps_i^{(n + 1)}} &\, \,
\begin{cases}
f  \brac{ x_0 , \ldots ,[ x_i , x_{ i+1 } , x_{ i+2} ] , \ldots , x_{ n + 1 }   }
& \text{for $i = j - 1 $} \\
f  \brac{ x_0 , \ldots ,[ x_i , x_{ i+1 } ] , \ldots, [ x_j , x_{ j+1 } ]
\ldots , x_{ n + 1 }   }
& \text{for $ i < j-1$ .}
\end{cases}
\end{align*}
\end{proof}

\subsection{ Degeneracy Maps  $\degenmaps^{(j , n)}$ and the Simplicial
Relations }

There are obvious reductions for
$ \mathscr S_h (f) \brac{ \rho_1 \otimes  \cdots \otimes  \rho_n }$
when some of the factors are equal to $1$, or at least, they commutes with
$h$ and other $\rho$'s.
Such redundancy of rearrangement can be removed by the degeneracy transformations
on $f$ that reduces the number of its arguments  by one.
\begin{lem}
\label{lem:degenmaps-j-for-h}
When $\rho_{ j+1} = 1$ in 
$  \rho_1 \otimes \cdots  \otimes  \rho_n  $, 
where $j = 0, \ldots , n-1$, we have
\begin{align*}
\mathscr S_h ( f)
\brac{ \rho_1 \otimes \cdots  \otimes  \rho_n \big |_{ \rho_{ j+1} =1}   }
=
\mathscr S_h \brac{ \degenmaps^{(j,n)} (f) }
\brac{ \rho_1 \otimes \cdots  \otimes \widehat{\rho_{ j+1} } \otimes
\cdots  \otimes     \rho_n } ,
\end{align*}
where $ \widehat{ \bullet }$ means $ \rho_{ j+ 1} $ is removed
and $\degenmaps^{(j, n)}: \funsp (n) \to \funsp( n-1 )$
is the restriction map  onto the hyperplane 
$ \set{ (x_0 , \ldots , x_n) : x_j = x_{ j+1}} \subset \R^{n+1}$:
\begin{align}
\label{eq:degenmaps-j-for-h}
\degenmaps^{(j , n)}(f) (x_0, \ldots, x_{ n-1}) =
f (x_0, \ldots, x_j, x_j, x_{ j+1}, \ldots, x_{ n-1}  ) .
\end{align}
\end{lem}
\begin{proof}
According to \cref{eq:scr-S-f-of-rho} with $ \rho_{ j +1 } = 1$,
\begin{align*}
&\, \,
f \brac{ h^{(0)}, \ldots, h^{(n)} } \brac{
\rho_1 \otimes \cdots  \otimes  1 \otimes  \cdots \otimes  \rho_n
} \\
= &\, \,
\int_{ \R^{n+1} } \widehat{ f} (\xi)
e^{\xi_0 h} \rho_1 e^{\xi_1 h} \cdots \rho_{ i-1}
e^{\xi_j h} \cdot  1 \cdot e^{\xi_{ j+1}  h} \rho_{ j+1} \cdots
\rho_n e^{\xi_n h}
\\
= &\, \,
\int_{ \R^{n+1} } \widehat{ f} (\xi)
e^{\xi_0 h^{(0)}} \cdots  e^{ \xi_j   h^{(j)}}
e^{  \xi_{ j+1}  h^{(j)}} \cdots  e^{  \xi_{n}  h^{(n-1)}}
\brac{ \rho_1 \otimes \cdots  \otimes \widehat{\rho_{ j+1} }
\otimes   \cdots  \otimes     \rho_n
} \\
= &\, \,
f \brac{ h^{(0)}, \cdots , h^{(j)}, h^{(j)}, h^{(j+1)}, \ldots, h^{(n-1)} }
\brac{ \rho_1 \otimes \cdots  \otimes \widehat{\rho_{ j+1} }
\otimes   \cdots  \otimes     \rho_n
} .
\end{align*}
The desired transformation of
$\degenmaps^{(j,n)} (f)$ in \cref{eq:degenmaps-j-for-h} is taken to be
the underlying spectral function shown in the last line above.
\end{proof}

\begin{prop}
\label{prop:spl-rlt-sig-i-j-prf}
In addition to the relations in \cref{prop:spl-rlt-del-i-j-prf}, we have
\begin{align}
\label{eq:spl-rlt-sig-i-j-prf}
\degenmaps_j^{(n -1)} \degenmaps_i^{(n)}  =
\degenmaps_i^{(n -1)} \degenmaps_{ j + 1} ^{(n)} ,
\, \, \, \,  \text{ for $ i < j$.}
\end{align}
\end{prop}
\begin{proof}

Observe that the right hand side of \cref{eq:degenmaps-j-for-h} is the ``dual'' of
the degeneracy maps $\sigma_j^{(n)}: [n] \to [n-1]$ which hits $j$ twice
in the simplicial category $\Delta$:
\begin{align*}
\sigma_j^{(n)}: 0 \mapsto 0, \ldots , j \mapsto j , j+1 \mapsto  j,
j+2 \mapsto  j+1 , \ldots , n \mapsto  n-1.
\end{align*}
In more detail, similar to the computation in \cref{prop:spl-rlt-del-i-j-prf},
we can rewrite \cref{eq:degenmaps-j-for-h} as:
\begin{align*}
f (x_0 , \ldots , x_n ) \xrightarrow{ \degenmaps_j^{(n)}}
f \brac{ x_{ \sigma_j^{(n)} (0)} , \ldots , x_{ \sigma_j^{(n)} (n)}   } .
\end{align*}
Consequently, composition of $ \degenmaps_j^{(n -1)} \degenmaps_i^{(n)}$
becomes the composition of $ \sigma_j^{(n + 1)} \sigma_i^{(n)}$ on
the subscripts of the variables $x_j$:
\begin{align*}
f (x_0 , \ldots , x_n )
\xrightarrow{   \degenmaps_j^{(n +1)}  \degenmaps_i^{(n)} }
f \brac{ x_{  \sigma_j^{(n + 1)} \sigma_i^{(n)} (0)} , \ldots ,
x_{ \sigma_j^{(n + 1)} \sigma_i^{(n)} (n)}   } .
\end{align*}
Similar result holds for $  \degenmaps_i^{(n -1)} \degenmaps_{ j + 1} ^{(n)}  $.
Therefore \cref{eq:spl-rlt-sig-i-j-prf} follows from the relations
$ \sigma_j^{(n + 1)} \sigma_i^{(n)} = \sigma_i^{(n + 1)} \sigma_{ j+1} ^{(n)}$
in the simplicial category $\Delta$.
\end{proof}

\subsection{
Face Maps $\facemaps^{(j, n)}$ from Variation
}
\label{sec:facemaps} 

Let $h \in \mathscr A$ be self-adjoint element and $f \in C^\infty(U) $ be
a smooth function  in a neighborhood of $\mathrm{Spec} (h) \subset U$.
By viewing $h = h^{(0)} ( 1_{ \mathcal A} )$
as the left-multiplication operator evaluated at the identity,
the functional calculus in \cref{eq:scr-S-f-defn} defines
an element $f (h) \defeq \mathscr S_h (f) ( 1_{ \mathcal A} )\in  \mathscr A$
in the algebra.
\begin{prop}
\label{prop:expan-f-h+b}
For $b \sim 0 \in \mathscr A$ self-adjoint, we have
the Taylor expansion around $h$:
\begin{align}
\label{eq:expan-f-h+b}
f ( h +b) \sim_{ b \searrow 0}
f(h) + \sum_{ n= 1 }^\infty
\mathscr S_h ( \ddif^n (f) )
( b \otimes \cdots  \otimes  b) ,
\end{align}
where $  b \otimes \cdots  \otimes  b = b^{\otimes  n}$.
\end{prop}
\begin{proof}
We refer to \cite[Prop. 3.7]{leschdivideddifference}.
\end{proof}

Let $\nabla : \mathscr A \to \mathscr A$ be a derivation and
$\alpha_t : \mathscr A \to \mathscr A$, $t \in  \R$,
be the associated one-parameter group of automorphisms. Namely,
$ \nabla a = \frac{d}{dt} \big |_{ t =0} \alpha_t (a)$
, for any $a \in  \mathscr A$ whenever the derivative exists.
They can be lifted to the tensor product $ \mathscr A^{\otimes n}$ and to the
functional calculus in the usual way:
\begin{align}
\label{eq:nabla-on-rho}
    &\, \,
\nabla \brac{ \rho_1 \otimes  \cdots \otimes   \rho_n }
= \frac{d}{dt} \bigg |_{ t =0}
\alpha_t ( \rho_1) \otimes \cdots \otimes \alpha_t ( \rho_n) ,
\\
    &\, \,
\label{eq:nabla-on-Sf}
\nabla \brac{ \mathscr S_h (f) } =
\frac{d}{dt} \bigg |_{ t =0}
\mathscr S_{   \alpha_t (h)} (f ) ,\, \,
\end{align}

For smooth enough $h$ (w.r.t. $\nabla$), we have
$\alpha_t (h) = h + b(t)$
where $ b (t) \sim_{ t \mapsto 0}
\nabla (h) t + \nabla^2 (h) t^2 /2 + \cdots $.
The expansion in \cref{eq:expan-f-h+b} implies:
\begin{align*}
\mathscr S_{ \alpha_t (h)} (f)  =
f (h) + t \mathscr S_h ( \ddif (f) ) ( \nabla h) + O (t^2) .
\end{align*}
Hence $\nabla \brac{  \mathscr S_{ \alpha_t (h)} (f)   } =
\mathscr S_h (\ddif (f)) ( \nabla h)$. The simplest example is the case
$f(x) = x^2$, then  $ \ddif (f) (x_0 , x_1) = x_0 + x_1$, which agrees with the
result from the Leibniz rule:
$\nabla ( h^2) = h \nabla (h) + \nabla (h) h
= \brac{ h^{(0)} + h^{(1)} } ( \nabla h )$.

Now, let us replace  $h$
by the left or the right multiplication,
that is  $h^{(0)}$ or $h^{(1)}$,
similar argument gives: for any $\rho \in  \mathscr A$,
\begin{align}
\label{eq:nabla-S-f-1var}
\nabla \brac{  f(h^{(0)})}  (\rho) =
\mathscr S_h ( \tilde f_{ \mathrm{lt}} ) ( \nabla(h) \otimes  \rho)
, \, \, \, \,
\nabla \brac{  f(h^{(1)})}  (\rho) =
\mathscr S_h ( \tilde f_{ \mathrm{rt}} ) (\rho \otimes  \nabla(h)),
\end{align}
where $ \tilde f_{ \mathrm{lt}} \brac{ x_0 , x_1 , x_2 }  =
\ddif (f) (x_0  , x_1)$ and
$ \tilde f_{ \mathrm{rt}} \brac{ x_0 , x_1 , x_2 }  =
\ddif (f) (x_1  , x_2) $.
The variation of $ \mathscr S_h (f)$ for
general $(n+1)$-variable spectral functions
$f( x_0 , \ldots , x_{ n} )$ can be computed following the Leibniz rule,
that is, we can differentiate $\alpha_t(h)^{(j)}$ one by one from $j=0$ to  $n$.
There is also a shift of indices needed to be taken care of
as shown in the case of $ \tilde f_{ \mathrm{rt}}$ define as above.
We summarize the result in a proposition below.
\begin{prop}
\label{prop:face-j-for-h}
For a fixed self-adjoint $h \in \mathscr  A$ and let $f \in \funsp(n)$ be
a spectral function with $(n+1)$ arguments.
The corresponding derivation $\nabla \brac{ \mathscr S_h (f) }$ defined in
\cref{eq:scr-S-f-of-rho}
is given by
\begin{align}
\label{eq:nabla-S-f}
\nabla\brac{ \mathscr S_h (f) } =
\sum_{ j=0 }^n
\mathscr S_h \brac{   \facemaps^{(j , n)} (f) }
\circ \nabla(h)_{(j , n)},
\end{align}
where  $ \nabla(h)_{(j , n)} : \mathscr A^{\otimes  n} \to \mathscr A^{\otimes
n+1}$ is the operator  of inserting  $\nabla(h)$ at the  $j$-th slot\footnote{
Note that we put $(j,n)$ in subscript to distinguish from the notation
$  \nabla(h)^{(j , n)} $ in \cref{eq:h-j-n-mlt-jslot} whose target space is
$ \mathscr A$.
},
see \cref{eq:nabla-S-f-example}.
The transformations
$\facemaps^{(j , n)}(f) : \funsp(n) \to \funsp(n+1)$
is given by applying divided difference at the  $j$-th argument with
a shift of indices after $j$:
$x_l \to x_{ l+1} $ for $ j+1 \le l \le  n$,
\begin{align}
\label{eq:face-j-for-h}
\facemaps^{(j , n)}(f) (x_0 , \ldots , x_{ n+1} ) =
f(x_0 , \ldots, \bullet, x_{ j+2} \ldots , x_{ n+1} ) 
[ x_j , x_{ j+1} ],
\, \,
\end{align}
where $\bullet$ is at the  $j$-th slot, $ 0\le  j \le  n$.
\end{prop}

\begin{proof}
\cref{eq:nabla-S-f} is an extension  of \eqref{eq:nabla-S-f-1var} from single  
to multivariable calculus. In fact, it is obtained  by differentiating in $t$ 
trough the multivariable function 
$ f ( \alpha_t (h^{(0)}) , \ldots , \alpha_t (h^{(n)}) )$ 
in \eqref{eq:nabla-on-Sf}, 
while the computation of the partial derivatives
$ (\partial_{ x_j} f) ( d x_j / dt )$ can be reduced in the single variable 
scenario in \eqref{eq:nabla-S-f-1var}. 
\end{proof}

Some remarks on notations:
\begin{itemize}

\item
The meaning of \eqref{eq:nabla-S-f} is better explained by examples.
For $f = f (x_0 , x_1  ,x_2)$, the evaluation of the r.h.s of
\eqref{eq:nabla-S-f} at $ \rho_1 \otimes  \rho_2 \in  \mathscr A^{\otimes  2}$
is given by:
\begin{align}
\label{eq:nabla-S-f-example}
&\, \,
\nabla \brac{
\mathscr S_h (f)
} ( \rho_1 \otimes  \rho_2) \\
= &\, \,
\mathscr S_h ( \facemaps^{(0,2)} (f))
\brac{ \nabla (h) \otimes  \rho_1 \otimes  \rho_2 }
+
\mathscr S_h ( \facemaps^{(1,2)} (f))
\brac{  \rho_1 \otimes \nabla (h)  \otimes  \rho_2 }
\nonumber \\
+  &\, \,
\mathscr S_h ( \facemaps^{(2,2)} (f))
\brac{  \rho_1 \otimes   \rho_2 \otimes  \nabla (h)} .
\nonumber
\end{align}

\item 
One can use another notation mentioned in \cref{sec:ddiff}
for the divided difference in \eqref{eq:face-j-for-h}:
\begin{align}
\label{eq:sbrac-as-ddif}
f\brac{ x_0 , \ldots , [x_j , x_{ j+1}] , x_{ j+2} ,
\ldots ,x_{ n+1}    }  .
\end{align}
Note that the $(n+1)$-tuple
$ \brac{ x_0 , \ldots , [x_j , x_{ j+1}] , x_{ j+2} , \ldots ,x_{ n+1}    }$
describes a map $\psi$ from $[n+1] = \set{0, \ldots , n+1}$ 
to $ [n] = \set{0, \ldots , n}$ (the $n$-slots)
by assigning the $(n+2)$ variables labeled by $x$ to their positions 
in the  $(n+1)$-tuple. 
It is a non-decreasing map  hitting $j \in  [n]$ twice, that is
the pre-image of the stationary point
$ \psi^{-1} (j) = \set{ x_j , x_{ j+1}}$ is enclosed in the
squared brackets. 
With regards to the notations in \cref{sec:cyc-DelC},
$\psi$ is nothing but the degeneracy map
$ \sigma_j^{[n+1]}: [n+1]  \to [n]  $
in the simplicial category $\Delta$. This observation is useful when 
computing the iterations of the face maps $\facemaps^{(j , n)}$, 
the results will be of the form akin to  \eqref{eq:sbrac-as-ddif}:
\begin{align*}
    f\brac{ \ldots , 
    [ \bullet , \ldots , \bullet], \ldots ,  
    [ \bullet , \ldots , \bullet], \ldots   
} 
\end{align*}
that gives rise to a $\psi: [m] \to [n]$ in $\Delta$, where $m , n$ are the
number of variables and slots  respectively. The squared brackets $[ \ldots ]$
enclose the stationary points of $\psi$ and indicate iterated divided
difference on the function $f$.
\end{itemize}

\subsection{Compatibility between $\facemaps^{(j , n)}$ and $\degenmaps^{(j ,n)}$}
\label{sec:face-deg-cmpt}

Unfortunately, we do not have the full simplicial structure, which has
something to do with the basic fact in calculus that differentiation and
evaluation do not commute.
\cref{prop:face-deg-cmpt} consists of relations that are taken from cyclic theory 
while \cref{prop:face-deg-cmpt-i-i} states the differences. 
\begin{prop}
\label{prop:face-deg-cmpt}
We have
\begin{align}
\label{eq:face-deg-cmpt}
\pmb\sigma_j \pmb\delta_i =
\begin{cases}
\pmb\delta_i \pmb\sigma_{ j-1}  & \text{ for $i<j-1 $,}\\
\pmb\delta_{ i-1} \pmb\sigma_j  & \text{ for $i > j+1$.}
\end{cases}
\end{align}
\end{prop}
\begin{proof}
We prove the first one
$ \pmb\sigma_j \pmb\delta_i = \pmb\delta_i \pmb\sigma_{ j-1}$ and leave the
other to the reader  since the calculations are quite similar.
For $f = f (x_0  , \ldots , x_n) \in  \funsp(n)$,
\begin{align*}
f \xrightarrow{\facemaps_i^{(n)}}
f \brac{ x_0, \ldots , [x_i , x_{ i+1} ] , \ldots , x_{ n+1}  }.
\end{align*}
With $j > i +1$, the result of $ \degenmaps_j^{(n+1)}\facemaps_i^{(n)} (f)$
is given by:
\begin{align*}
f \brac{ x_0, \ldots , [x_i , x_{ i+1} ] , \ldots , x_{ n+1}  }
\xrightarrow{ \degenmaps_j^{(n + 1)}}
f \brac{ x_0, \ldots , [x_i , x_{ i+1} ] , \ldots , x_j , x_j , \ldots , x_{ n}  }.
\end{align*}
For the other side,
\begin{align*}
&\, \,
f \xrightarrow{  \degenmaps_{ j-1} ^{(n)}}
f \brac{ x_0, \ldots , x_{ j-1}  , x_{ j-1}  , \ldots, x_{ n-1}   }   \\
\xrightarrow{ \degenmaps_j^{(n + 1)}} &\, \,
f \brac{ x_0, \ldots , [x_i , x_{ i+1} ] , \ldots , x_j , x_j , \ldots , x_{ n}  }.
\end{align*}
For the second step, since the position of $ [x_i , x_{ i+1} ]$ is before  $j-1$,
there is a shift on the subscript of variables $x_l$ for all  $l>i+1$ which
gives rise to  $ (x_j , x_j)$.
The two sides indeed agree.
\end{proof}

The uncovered cases in \cref{eq:face-deg-cmpt} are $\degenmaps_j \facemaps_j$
and $\facemaps_j \degenmaps_j$ ( set $i= j -1$ or  $i = j+1$ for the right hand
side of \cref{eq:face-deg-cmpt}).

\begin{prop}
\label{prop:face-deg-cmpt-i-i}
We have
\begin{align}
\label{eq:face-deg-cmpt-i-i-I}
\pmb\sigma^{(i,n+1)} \pmb\delta^{(i,n)} &= \partial_{ x_i } ,
\\
\label{eq:face-deg-cmpt-i-i-II}
\pmb\delta^{(i,n+1)} \pmb\sigma^{(i,n)} &=
\pmb\sigma^{(i+1 , n+1)} \pmb\delta^{(i,n)} +
\pmb\sigma^{(i,n+1)} \pmb\delta^{(i+1 , n)} ,
\end{align}
where $\partial_{ x_i } $ is the partial derivative acting on 
$ f (x_0 , \ldots x_n ) \in  \funsp (n)$.
\end{prop}
\begin{proof}
\cref{eq:face-deg-cmpt-i-i-I}  follows from  the confluent version of
divided difference (cf. \cref{eq:ddif-confluet}):
\begin{align*}
\pmb\sigma^{(i,n+1)} \pmb\delta^{(i,n)} (f) (x_0, \ldots, x_n)
&=
\pmb\delta^{(i,n)} (f) (x_0, \ldots, x_i, x_i, \ldots, x_n) \\
& =
f \brac{   x_0, \ldots, [ x_i, x_i ], \ldots, x_n} \\
&=
\brac{ \partial_{ x_i}  f} (x_0, \ldots, x_n) .
\end{align*}
To prove \cref{eq:face-deg-cmpt-i-i-II}, let us start with the left hand side
\begin{align*}
f (x_0 , \ldots , x_{ n+1} ) \xrightarrow{ \degenmaps^{(i,n+1)} }
f \brac{ x_0 , \ldots ,x_i , x_i , \ldots , x_n } .
\end{align*}
To continue, we apply \cref{eq:f-z-z-divided} to compute the
divided difference of the function
$x_i \mapsto  f \brac{ \ldots , x_i , x_i , \ldots }$ as above,
\begin{align*}
&\, \,
f \brac{ x_0 , \ldots ,x_i , x_i , \ldots , x_n } \\
\xrightarrow{ \facemaps^{(i,n)}} &\, \,
f \brac{ x_0, \ldots, x_i , [x_i , x_{ i+1} ] , \ldots, x_n  } +
f \brac{ x_0, \ldots, [x_i , x_{ i+1} ] , x_{ i+1},   \ldots, x_n  } ,
\end{align*}
where the two terms are equal to 
$\degenmaps^{(i,n+1)} \facemaps^{( i+1,n)} ( f ) $ and 
$ \degenmaps^{( i+1,n+2)} \facemaps^{(i,n)} (f)$ respectively. 
Let us compute $\degenmaps^{(i,n+1)} \facemaps^{( i+1,n)} ( f ) $
 and leave the verification of the second one to the reader. 
Indeed, 
\begin{align*}
    &\, \,
f \xrightarrow{ \facemaps^{( i+1,n)}}
f \brac{ x_0 , \ldots , x_i , [x_{ i+1} , x_{ i+2} ] , x_{ i+3}
\ldots , x_{ n+1}    } 
 \\ 
\xrightarrow{ \degenmaps^{(i,n+1)}} &\, \,
f \brac{ x_0 , \ldots , x_i , [x_{ i} , x_{ i+1} ] , x_{ i+2}  \ldots , x_{ n} } .
\end{align*}
\end{proof}

\begin{lem}
Consider the function $z \mapsto  f(z, z)$ induced from a two-variable
function $f$, its divided difference is given by:
\begin{align}
\label{eq:f-z-z-divided}
f(z,z)[x,y]_z
& =
f(x, z) [x,y]_z  + f(z,y) [x,y]_z   \\
& =
f(y, z) [x,y]_z  + f(z,x) [x,y]_z .
\nonumber
\end{align}
\end{lem}
\begin{proof}
The computation is straightforward:
\begin{align*}
f(x, z) [x,y]_z  + f(z,y) [x,y]_z
&=
\frac{ f(x,x) - f (x,y) }{x-y} +
\frac{ f(x,y) - f (y,y) }{x-y} \\
&=
\frac{ f(x, x) - f (y , y)}{x - y}
=          f(z,z)[x,y]_z .
\end{align*}
\end{proof}

\subsection{Tracial Functionals and the Cyclic Operators $\cyc_{(n)}$}
\label{sec:tr-tau} 

After the discussion of differentiation, let us now further assume that
the algebra $\mathscr A$ admits a tracial functional
$\varphi_0 : \mathscr A \to \mathbb{C} $ playing the role of integration.
The trace property has its usual meaning:
$ \varphi_0 ( \rho \rho' ) = \varphi_0 ( \rho ' \rho)$,
$\forall \rho, \rho' \in \mathscr A$.
In particular, when applying $\varphi_0$ to the integrand of r.h.s of
\cref{eq:scr-S-f-of-rho}, we see that
\begin{align*}
\varphi_0
\brac{ e^{\xi_0 h} \rho_1 e^{\xi_1 h} \rho_2 e^{\xi_2 h}} =
\varphi_0
\brac{   e^{(\xi_0 + \xi_2) h} \rho_1 e^{\xi_1 h} \rho_2},
\end{align*}
which leads to another reduction
$\degenmaps^{( n , n)} : \funsp(n)  \to \funsp (n-1)$:
\begin{align}
\label{eq:extr-deg-n-n}
\degenmaps^{ (n , n) }  (f) (x_0, \ldots, x_{ n-1} ) =
f(  x_0, x_1, \ldots,   x_{ n-1}  , x_0 ) .
\end{align}

\begin{lem}
\label{lem:extr-deg-n-n-full}
The extra degeneracy  $\degenmaps^{(n , n)}$ defined above
\cref{eq:extr-deg-n-n}
is responsible for
\begin{align}
&\, \,
\varphi_0\brac{
\mathscr S_h (f)
\brac{
\rho_1 \otimes \cdots  \otimes  \rho_n
}
} \nonumber \\
= &\, \,
\varphi_0\brac{
\mathscr S_h \brac{ \degenmaps^{( n , n)} (f) }
\brac{ \rho_1 \otimes \cdots  \otimes  \rho_{ n-1} } \cdot \rho_n
}.
\label{eq:extr-deg-n-n-full}
\end{align}

\end{lem}
\begin{rem}
We shall see in \cref{eq:extr-deg-chk} that
$\degenmaps^{(n,n)}  = \degenmaps^{(0,n)} \cyc_{(n-1)}^{-1 }$,
which corresponds to the extra degeneracy \cref{eq:extr-deg-cyc}
in the cyclic category $\Delta C$.
\end{rem}
\begin{proof}
We first apply $\varphi_0$ to the right hand side of
\cref{eq:scr-S-f-of-rho} and then use the trace property to
move $ e^{\xi_n h}$ to the very left:
\begin{align*}
&\, \,
\varphi_0\brac{
f \brac{ h^{(0)}, \ldots, h^{(n)} } \brac{
\rho_1 \otimes \cdots  \otimes  \rho_n
}
} =
\varphi_0 \brac{
\int_{ \R^{n+1} } \widehat{ f} (\xi)
e^{\xi_0 h} \rho_1 e^{\xi_1 h} \cdots \rho_n e^{\xi_n h}
}
\\
= &\, \,
\varphi_0 \brac{
\int_{ \R^{n+1} } \widehat{ f} (\xi)
e^{\xi_n h} e^{\xi_0 h} \rho_1 e^{\xi_0 h} \cdots \rho_{ n-1} e^{\xi_{ n-1}  h}
\rho_n
}
\\
= &\, \,
\varphi_0\brac{
f \brac{ h^{(0)}, \ldots, h^{(n-1)} , h^{(0)}} \brac{
\rho_1 \otimes \cdots  \otimes  \rho_{ n-1}
}  \cdot  \rho_n}  ,
\end{align*}
where the Fourier transform integral  in the middle line  is exactly the
function in the right hand side of \cref{eq:extr-deg-n-n}.
\end{proof}

The trace property of the functional $\varphi_0$ allows cyclic permutations on
the $\rho$-factors appeared in the local expression \cref{eq:extr-deg-n-n-full}.
\begin{prop}
\label{prop:tau-n-defn}
Given $f (x_0 , \ldots ,x_n) \in \funsp(n)$ and
$(n+1)$ elements
$\rho_1, \ldots , \rho_{ n+1}  \in \mathscr A$,
we have the following cyclic permutation:
\begin{align*}
\varphi_0 \brac{
\mathscr S_h (f)
\brac{ \rho_1 \otimes  \cdots \otimes  \rho_n }
\cdot \rho_{ n+1}
}   =
\varphi_0 \brac{
\mathscr S_h \brac{ \cyc_{(n)} (f)}
\brac{ \rho_2 \otimes  \cdots  \otimes \rho_{ n+1}  }
\cdot  \rho_1
} ,
\end{align*}
where the transformation
$\cyc_{(n)}: \funsp(n) \to \funsp(n)$ is
the corresponding cyclic permutation on arguments of $f$:
\begin{align}
\label{eq:tau-n-defn}
\cyc_{(n)} (f) (x_0 , \ldots , x_n)  =
f ( x_n, x_0,\ldots, x_{ n-1} )  .
\end{align}
\end{prop}

\begin{proof}
Again, the key is to look at the following part of the Fourier
transform in \cref{eq:scr-S-f-of-rho} in which one can cyclic permute the
factors of the product inside $\varphi_0$ according to the trace property:
\begin{align*}
&\, \,
\varphi_0
\brac{ e^{i \xi_0 h } \rho_1  e^{i \xi_1 h }
\cdots \rho_n e^{i \xi_n h }\cdot  \rho_{ n+1}
}  =
\varphi_0
\brac{  e^{i \xi_1 h } \rho_2
\cdots \rho_n e^{i \xi_n h }\cdot  \rho_{ n+1}
e^{i \xi_0 h } \rho_1
} \\
= &\, \,
\varphi_0\brac{
\brac{
e^{i \xi_1 h^{(0)} } \cdots
e^{i \xi_{ n}  h^{(n-1)} }
e^{i \xi_0 h^{(n)} }
}
\brac{ \rho_2 \otimes  \cdots  \otimes  \rho_{ n+1}  }
\cdot \rho_1
} \\
= &\, \,
\varphi_0 \brac{
\mathscr S_h \brac{ e^{ i \xi_0 x_n
+ i \xi_1 x_0 + \cdots  + i \xi_n x_{ n-1}  } }
}
\brac{ \rho_2 \otimes  \cdots  \otimes  \rho_{ n+1}  }
\cdot \rho_1 .
\end{align*}
The last line gives rise to  the function $ f(x_n , x_0 , \ldots , x_{ n-1} )$
after Fourier transform (the coefficient of $\xi_j$ corresponds to the
$j$-th argument of  $f$).
\end{proof}

\subsection{Compatibility between Cyclic and Simplicial Structures}
\label{sec:cyclic-vs-smpl-rel}

\begin{prop}
\label{prop:cyclic-vs-smpl-rel}
We have the compatibility relations between the cyclic maps and
the degeneracy maps:
\begin{align}
\label{eq:cyclic-rel-sig-lem}
\pmb\tau_{(n)} \degenmaps^{(i ,n+1)}  =
\degenmaps^{( i-1 ,n+1)}  \cyc_{(n+1)}  , \, \,
1\le i \le  n.
\end{align}
In particular, $ \pmb\tau_{ (n)}  \pmb \sigma^{(0,n+1)}
= \pmb \sigma^{(n , n)} \pmb\tau^2_{ (n+1) } $.
Similarly, for the face maps
\begin{align}
\label{eq:cyclic-rel-tau-lem}
\pmb\tau_{ (n)}  \facemaps^{(i , n-1)} =
\facemaps^{(  i-1 , n-1)} \pmb\tau_{(n-1)}, \, \,
i=1, \ldots,n .
\end{align}
It follows that $\facemaps^{(n , n)} = \pmb\tau_{(n)} \facemaps^{(0, n-1)} $.
\end{prop}
\begin{proof}
The computation is elementary, we check \cref{eq:cyclic-rel-tau-lem} as an
example and leave \cref{eq:cyclic-rel-sig-lem} to the reader.
Begin with the  left hand side of \cref{eq:cyclic-rel-tau-lem}:
\begin{align*}
\pmb\tau_{ (n)}  \facemaps_i^{(n-1)}
(x_0, \ldots , x_n)
&=
\facemaps_i^{(n-1)}
(x_n , x_0, \ldots , x_{ n-1} ) \\
&=
f \brac{  x_n , x_0, \ldots, [x_{ i-1} , x_{ i}], \ldots,x_{ n-1}  }
\end{align*}
and then the right hand side:
\begin{align*}
\facemaps_{ i-1}^{( n )}
\pmb\tau_{ (n-1) }  (f) (x_0, \ldots , x_n)
&=
\pmb\tau_{ ( n-1 ) }  (f) (x_0, \ldots , x_n)
\brac{     x_0, \ldots, [x_{ i-1} , x_{ i}], \ldots,x_{ n}  } \\
&=
f \brac{ x_n, x_0, \ldots, [x_{ i-1} , x_{ i}], \ldots,x_{ n-1}  } .
\end{align*}
\end{proof}

\section{Main Results}
\label{sec:main} 
Let us denote by $\mathcal C$ the category with objects
$ \funsp (n)$, $n=0,1, \ldots,$ and morphisms
generated by the transformations 
$\set{\facemaps^{(j, n)}  , \degenmaps^{(j, n)},
\cyc_{ (n)} }_{ n=0}^{\infty} $  introduced in 
\cref{prop:face-j-for-h}, \cref{lem:degenmaps-j-for-h} and 
\cref{lem:degenmaps-j-for-h}: 
\begin{align*}
    &
 \facemaps^{(j,n)} : \funsp(n) \to \funsp(n  + 1 ), \, \, 0\le  j \le n   , 
 \\
    &
  \degenmaps^{(i,n)}: \funsp(n)  \to  \funsp(n -1  ) \, \, 0\le  j \le n   , 
  \\
    &
  \degenmaps^{(i,n)}: \funsp(n)  \to  \funsp(n -1  ) .
\end{align*}
The new contributes of the paper begins with the observation that they almost
fulfil the generating relations of the cyclic category $\Delta C$
in  \cref{defn:DelCyc}.
To this end, one needs to add  one more set of  face maps 
\begin{align}
\label{eq:last-face-del-n+1}
\facemaps^{(  n+1 , n)}  \defeq \tau_{ (n + 1)}  \facemaps^{( 0 , n)}:
\funsp(n) &\to \funsp(n  + 1 )  \\
f\brac{ x_0 , \ldots , x_n } &\mapsto
f\brac{ [x_{ n+1 } , x_0] , x_1 , \ldots , x_n} .
\nonumber
\end{align}
Unlike the cyclic category $\Delta C$,  the morphism set of $ \mathcal C$
contains all partial derivatives, due to the failure  
of the compatibility between the simplicial and co-simplicial structures.
 This feature   has been predicted in, for instance,
the $a_4$-term computation \cite[\S~4]{2016arXiv161109815C} and 
deserves further investigation. 

Strictly speaking, (the objects of) $ \mathcal C$  depends
on  the choice of the underlying algebra $ \mathscr A$ and
the self-adjoint $h \in \mathscr A$.
Nevertheless, the relations presented in below are universal.
\begin{thm}
\label{thm:cyc-cato-sim-I}
With slight simplification on notations, such as
$\facemaps^{ (j, n)} \mapsto  \facemaps_j$, 
the generators of the category $ \mathcal C$ defined above
fulfil the simplicial and co-simplicial relations\textup{:}
\begin{align}
\label{eq:facj-faci-thm}
& \pmb\delta_j  \pmb\delta_i = \pmb\delta_i \pmb\delta_{ j-1}, \, \,
\text{ for $i<j$} \\
& \pmb\sigma_j \pmb\sigma_i = \pmb\sigma_i  \pmb\sigma_{ j+1} , \, \,
\text{ for $i \le j$}
\label{eq:degj-degi-thm}
\end{align}
with a  compatible  cyclic structure\textup :
\begin{align}
\label{eq:tau-vs-f-d-thm}
\pmb\tau_n  \pmb\sigma_{ i} = \pmb\sigma_{ i-1} \pmb\tau_{ n +1} , \, \,
\pmb\tau_n  \pmb \delta_i = \pmb \delta_{  i-1} \pmb\tau_{ n-1}
.
\end{align}
But, the compatibility among simplicial structures is modified in the
following way\textup{:}
\begin{align}
\label{eq:degj-faci-thm}
\pmb\sigma_j \pmb\delta_i =
\begin{cases}
\pmb\delta_i \pmb\sigma_{ j-1} & \text{ for $i<j-1$,} \\
\facemaps_i \degenmaps_i  -
\pmb\sigma_{ i} \pmb\delta_{ i+1}
& \text{ for $i = j - 1$, } \\
\partial_{ x_i} & \text{ for $ i = j$,} \\
\pmb\delta_{ i-1}  \pmb\sigma_{ j} & \text{ for $i > j + 1$ ,}
\end{cases}
\end{align}
where $\partial_{ x_i} $ is the partial derivative acting on functions
$f(x_0 , \ldots , x_n) \in \funsp (n)$.
\end{thm}
\begin{rem}
Note that \cref{eq:degj-faci-thm} implies that
\begin{align*}
\pmb\sigma_j \pmb\delta_i = \facemaps_{ i-1 }  \degenmaps_{ i-1 }   -
\pmb\sigma_{ i} \pmb\delta_{ i-1 } ,
\, \, \text{ for $i = j + 1$. }
\end{align*}
\end{rem}

\begin{proof}
The verification is spread out among \cref{prop:spl-rlt-del-i-j-prf},
\cref{prop:spl-rlt-sig-i-j-prf},
\cref{prop:face-deg-cmpt},
\cref{prop:face-deg-cmpt-i-i},
and \cref{prop:cyclic-vs-smpl-rel}.
Concerning  the new face maps in \cref{eq:last-face-del-n+1},
which  are not covered in the aforementioned propositions,
they are defined in such way 
(according to \cref{eq:simp+cyc-rels-V-2-1}) that
\cref{eq:degj-degi-thm} and \cref{eq:tau-vs-f-d-thm} extend automatically.
\end{proof}

Thanks to the self-duality of the cyclic category $\Delta C$,
one can also choose to work with the generators of the opposite category
$\Delta C^{\mathrm{op}}$.
Let us implement the duality functor from $\Delta C^{\mathrm{op}}$
to $\Delta C$ discussed at the end of \S~\ref{sec:cyc-DelC},
starting with the extra degeneracies (see \cref{eq:extr-deg-cyc})
$  \degenmaps_0^{(n+1)} \cyc_{ ( n+1)}: \funsp (n+1) \to \funsp (n)$:
\begin{align}
\label{eq:extr-deg-chk}
\degenmaps_0^{(n+1)} \cyc_{ ( n+1)} (f)
\brac{ x_0 , \ldots , x_n }  =
\cyc_{ ( n+1)} (f)
\brac{ x_0 ,  x_0 , \ldots , x_n }  =
f \brac{ x_0 , x_1 , \ldots , x_n , x_0  } ,
\end{align}
which full agrees with
$ \degenmaps_{ n+1} ^{(n+1)}$ defined \cref{eq:extr-deg-n-n}.
Following \cref{eq:dul-futr-d-sig,eq:dul-futr-s-del}, we define
\begin{align}
\label{eq:cycop-gen-de}
&
\decop^{(n)}_j \defeq \degenmaps^{(j,n)} :
\funsp(n) \to \funsp(n-1),
\, \,  j = 0, \ldots , n,
\\
&
\label{eq:cycop-gen-fac}
\incop^{(n)}_j \defeq \facemaps^{( j+1,n)}:
\funsp(n) \to \funsp(n + 1),
\, \,  j = 0, \ldots , n,
\\
&
\label{eq:cycop-gen-tau}
\cycop_{ (n)} \defeq \cyc_{ (n)}^{-1} :\funsp(n) \to \funsp(n) .
\end{align}

\begin{cor}
\label{cor:op-cyc-cate-relts}
With the new set of  generators
$\set{\decop^{(n)}_j , \incop^{(n)}_j , \cycop^{(n)}_j}_{ n=0 }^\infty $,
the category $ \mathcal C$ almost form a
$\Delta C^{\mathrm{op}}$-module \textup (cf. \cref{defn:DelCyc-OP}\textup{):} 
\begin{enumerate}[$ (1)$]
\item  Pre-simplicial structures\textup :
\begin{align*}
& \decop_i \decop_j = \decop_{ j-1} \decop_i  ,
\, \,  \, \,  \text{ for $i<j$,} \\
& \incop_i \incop_j = \incop_{ j+1 } \incop_i ,
\, \,  \, \,  \text{ for $i \le  j$.}
\end{align*}

\item  Modified compatibility between face and degeneracy maps\textup :
\begin{align*}
\decop_i \incop_j =
\begin{cases}
\incop_{ j-1 } \decop_i ,
&  i<j , \\
\incop_{ j-1}  \decop_j - \decop_{ j+1 } \incop_{ j-1} ,
&  i = j , \\
\partial_{ x_{ j+1}  }
&  i = j + 1  , \\
\incop_{ j}  \decop_{ j+1} - \decop_{ j+1} \incop_{ j+1 }
&  i = j + 2  , \\
\incop_{ j } \decop_{  i - 1 }  ,
& i > j+2 ,
\end{cases}
\end{align*}
where $\partial_{ x_i } $ is the partial derivative on the $i$-th argument.
\item Cyclic relations: $\cycop_{ (n)} ^{n+1} = \mathrm{id}_{ (n)} $ and
\begin{align*}
\decop_i \cycop_{ (n)}  = \cycop_{ (n-1) }  \decop_{ i-1}
,\, \, \, \,
\incop_i  \cycop_{ (n)}  =  \cycop_{ ( n+1 ) } \incop_{ i+1}  ,
\end{align*}
where $ 1\le  i \le  n$. It follows that
\begin{align*}
\decop_0 \cycop_{ (n)}  = \decop_n, \, \, \, \,
\incop_0 \cycop_{ (n)}  = \cycop_{ ( n+1 )}^2 \incop_n .
\end{align*}
\end{enumerate}
\end{cor}

\section{Comparison with Hopf Cyclic Theory}
\label{sec:cmp-Hcyc}

Let us keep the notations introduced at the end of \cref{sec:rearr-op-fun}
and denote $ \mathcal O = C_c^\infty(U_h)$.
We will, in this section, replace the objects $ \funsp(n,h)$ of $ \mathcal C$
by the algebraic tensor product $ \mathcal O^{\otimes (n+1)}$ and then 
reproduce the simplicial and cyclic relations stated in 
\cref{thm:cyc-cato-sim-I}.
It turns out that the generators of $ \mathcal C$, when applied to functions of
separating variables: 
$f (x_0 ,  \ldots , x_{ n+1} ) = f_0 (x_0) \cdots f_{ n+1}  (x_{ n+1} )
\in  \mathcal O^{\otimes (n+1)}$,  
can be rewritten in  a more acquainted form which 
is in reminiscent of the counterpart in Hopf cyclic theory.   
Let us begin with the degeneracy and the
cyclic maps in \eqref{eq:degenmaps-j-for-h} and \eqref{eq:tau-n-defn}.
It is not difficult to see that they become
\begin{align}
\label{eq:deg_n-sp-var}
\degenmaps_j^{(n+1)}:
\mathcal O^{ \otimes (n+1)} &\to \mathcal O^{\otimes  n}   , \, \, \, \,
f_0 \otimes \cdots \otimes  f_{n+1}
\mapsto
f_0 \otimes \cdots \otimes  f_j f_{ j+1} \otimes \cdots \otimes  f_{ n+1 } ,
\end{align}
where $j=0, \ldots ,n$ and
\begin{align}
\label{eq:cyc_n-sp-var}
\cyc_{ (n)} :
\mathcal O^{ \otimes (n)} &\to \mathcal O^{\otimes  n}   , \, \, \, \,
f_0 \otimes \cdots \otimes  f_{n}
\mapsto
f_{ n} \otimes  f_0 \otimes \cdots  \otimes  f_{ n-1}     .
\end{align}
which are part of the generators of the cyclic module  of the
algebra $ \mathcal O$, cf. \cite[\S1.1,\S2.1]{Loday:1992vm}.
In particular, the relations in
\eqref{eq:spl-rlt-sig-i-j-prf} and \eqref{eq:cyclic-rel-sig-lem} follows
immediately from classical results.

Next, let us reexamine the relations in \eqref{eq:spl-rlt-del-i-j-prf}
and \eqref{eq:cyclic-rel-tau-lem}.
They rely only  on the following algebraic aspect of
the divided difference $f(x) \mapsto \ddif (f) (x_0 , x_1)$, namely,
as a coproduct map of the algebra $ \mathcal O$:
$\ddif : \mathcal O \to \mathcal O \hotimes \mathcal O$.
Note that the completion
of the algebraic tensor product is inevitable for the function $ (x_0
- x_1)^{-1}$ is not of separating variables. Nevertheless, let us still
make use of the Sweedler notation
\begin{align}
\label{eq:divdif-coprd-notn}
\ddif (f) = (f)_{ (1)} \hotimes (f)_{ (2)}  ,
\end{align}
where the complete tensor product $ \hotimes $ emphasize
that the underlying summation  might be an infinite one.
The composition rule \eqref{eq:Composition-rule} and the fact
that $\ddif (f)$ is symmetric in all the arguments leads to
the coassociativity $  \brac{ \ddif \otimes 1 }\ddif (f) =
\brac{ 1 \otimes  \ddif  }\ddif (f)$, that is
\begin{align}
\label{eq:divdif-swdl}
((f)_{ (1)})_{ (1)}  \hotimes   ((f)_{ (1)})_{ (2)}
\hotimes  (f)_{ (2)} =
(f)_{ (1)}  \hotimes
((f)_{ (2)})_{ (1)}  \hotimes   ((f)_{ (2)})_{ (2)}  .
\end{align}
Now the face operators can be written as
\begin{align}
\label{eq:fac_n-sp-var}
\facemaps^{(j,n)} = 1 \otimes \cdots \otimes  \ddif \otimes  \cdots \otimes 1:
\mathcal O^{ \hotimes  (n+1)}  & \to \mathcal O^{ \hotimes  (n+2)}  ,
\end{align}
where $\ddif$ appears at the  $j$-th factor, that is,
\begin{align}
\label{eq:fac_n-sp-var-otimes-fi}
f_0 \hotimes \cdots \hotimes f_n
\xrightarrow{ \facemaps^{(j,n)}}
f_0 \hotimes \cdots \hotimes  \ddif (f_j) \hotimes  \cdots \hotimes f_n  =
f_0 \hotimes \cdots \hotimes  (f_j)_{ (1)}  \hotimes (f_j)_{ (2)}
\hotimes \cdots \hotimes  f_n .
\end{align}
The action of the cyclic operators $\cyc$ in \eqref{eq:cyc_n-sp-var}
and face operators
resemble their counterparts  $\tau$ and $\delta_j$ described in
Hopf cyclic cohomology theory \cite[\S~7]{Connes:1998tp}.
To give another proof of the relations in 
\eqref{eq:spl-rlt-del-i-j-prf}  and \eqref{eq:cyclic-rel-tau-lem}, 
we point out that their the Hopf cyclic counterparts only require the
coassociativity of the underlying comultiplication,  and thus the arguments
there can applied without much modification.

An obvious reason for  not having all the relations in the cyclic
category $\Delta C$ is the fact that the coproduct $\ddif$ is not compatible
\footnote{The compatibility means the coproduct $\triangle$ is an
algebra homomorphism, that is
$ \triangle (f_1 f_2) = \triangle (f_1) \triangle ( f_2 )$. }
with the pointwise multiplication of $ \mathcal O$.
\begin{prop}
    With regards to the new definitions of the face
    and the degeneracy maps given in \eqref{eq:fac_n-sp-var-otimes-fi}
    and \eqref{eq:deg_n-sp-var}, the relations stated in 
    \textup{  \cref{prop:face-deg-cmpt}}   and 
    \textup{  \cref{prop:face-deg-cmpt-i-i} }
    still hold true.
\end{prop}
\begin{proof}
    One can see clearly in the proof of \cref{prop:face-deg-cmpt} that 
the multiplication
$f_i \otimes f_j \mapsto  f_i f_j$ and the divided difference
playing the role of coproduct do not interact with each other. 
Therefore the relations \eqref{eq:face-deg-cmpt} belong to the cyclic theory,
and thus as before, arguments in Hopf cyclic theory shall work without much
modification. We leave the details to the reader. 

Let us check \cref{prop:face-deg-cmpt-i-i} using the new notations.
Start with \eqref{eq:face-deg-cmpt-i-i-I}, the left hand side 
$ \facemaps^{(i , n+1 )}\degenmaps^{(i,n)}$ can be computed as follows:
\begin{align*}
&  f_0 \hotimes \cdots  \hotimes f_n
\xrightarrow{ \facemaps^{(i,n)}}
f_0 \hotimes \cdots \hotimes
\brac{ f_i }_{ (1)} \hotimes    \brac{ f_i }_{ ( 2)}
\hotimes \cdots \hotimes  f_n \\
\xrightarrow{ \degenmaps^{( i,n)} } &
f_0 \hotimes \cdots \hotimes
\brac{ f_i }_{ (1)} \brac{ f_i }_{ ( 2)}  \hotimes f_{ i+1}
\hotimes \cdots \hotimes  f_n, 
\end{align*}
which indeed agrees with the right hand side as 
$ \brac{ f_i }_{ (1)} \brac{ f_i }_{ ( 2)} = \partial_{ x_i }(f_i)$ according to 
\eqref{eq:ddif-confluet}.

To see the last one \eqref{eq:face-deg-cmpt-i-i-II}, we first have to convert
\eqref{eq:f-z-z-divided} by assuming $f (x, y) = f_1 ( x ) f_2 (y)$ and then 
apply the Leibniz rule in \eqref{eq:Leibniz-rule}, the new form in terms of the 
Sweedler notation looks like:
\begin{align*}
\ddif ( f_1 f_2 ) =
(f_1)_{ (1)} \hotimes (f_1)_{ (2)} f_2  +
f_1 (f_2)_{ (1)} \hotimes (f_2)_{ (2)} .
\end{align*}
Now we are ready to check \eqref{eq:face-deg-cmpt-i-i-II}. 
The two term on the right hand side are given by
\begin{align*}
&  f_0 \hotimes \cdots  \hotimes f_n
\xrightarrow{ \facemaps^{(i,n)}}
f_0 \hotimes \cdots \hotimes
\brac{ f_i }_{ (1)} \hotimes    \brac{ f_i }_{ ( 2)}
\hotimes \cdots \hotimes  f_n \\
\xrightarrow{ \degenmaps^{( i+1,n)} } &
f_0 \hotimes \cdots \hotimes
\brac{ f_i }_{ (1)} \hotimes    \brac{ f_i }_{ ( 2)}   f_{ i+1}
\hotimes \cdots \hotimes  f_n
\end{align*}
and similarly
\begin{align*}
&  f_0 \hotimes \cdots  \hotimes f_n
\xrightarrow{ \degenmaps^{(j,n)} \facemaps_{ i+1} ^{(n)}}
f_0 \hotimes \cdots \hotimes
f_{ i} \brac{ f_{ i+1}  }_{ (1)} \hotimes    \brac{ f_{ i+1} }_{ ( 2)}
\hotimes \cdots \hotimes  f_n .
\end{align*}
While, for the left hand side, we have
\begin{align*}
&  f_0 \hotimes \cdots  \hotimes f_n
\xrightarrow{ \facemaps^{(i,n)} \degenmaps^{(i,n)}}
f_0 \hotimes \cdots \hotimes
\ddif  (f_i f_{ i+1} ) \hotimes \cdots \hotimes  f_n
\\
= &
f_0 \hotimes \cdots \hotimes
(f_i)_{ (1)} \hotimes  (f_i)_{ (2)} f_{ i+1} \hotimes  \cdots
\hotimes f_n \\
+ &
f_0 \hotimes \cdots \hotimes
f_{ i} \brac{ f_{ i+1}  }_{ (1)} \hotimes    \brac{ f_{ i+1} }_{ ( 2)}
\hotimes  \cdots \hotimes f_n ,
\end{align*}
which is exactly the sum of the two terms obtained before.
\end{proof}

\section{Application to Modular Geometry on Noncommutative (Two) Tori}
\label{sec:appl-modgeo} 
The generators of $ \mathcal C$
are grown out of the explicit computations carried out
in the author's previous works \cite{Liu:2018ab,Liu:2018aa,Liu2020General-Rearran}, 
aiming at improving two analytic backbones behind  the spectral geometry:
the pseudo-differential and the variational calculus.

\subsection{Connes-Moscovici Type Functional Relations}
\label{sec:CM-relt}

Let $ \mathscr A = C^\infty(\mathbb{T}^2_\theta)$ be the smooth noncommutative
two torus generated by two unitary elements $U$ and  $V$ with the relation
$UV = e^{i \theta} V U$, for some irrational $\theta \in \R \setminus \Q$.
The algebra $ \mathscr A$ is $\Z^2$-graded and the components are
all one-dimensional spanned by $ U^n V^m \in \brac{ \mathscr A }_{n,m}$,
$(n,m) \in \Z^2$.
More precisely, $ \mathscr A$ consists of series
$ \sum_{ (n,m) \in \Z^2 } a_{ (n,m)} U^n V^m $ whose coefficients
$ a_{ n,m} \in \mathbb{C} $ are of rapidly decay in $(n,m)$.
Denote by $\nabla_1$ and $\nabla_2$ the two basic derivations
\footnote{The classical notation for the basis derivations $\delta_1$ and $\delta_2$
has been used  in the definition of the cyclic category $\Delta C$ in
\cref{sec:cyc-DelC}.  }
on $ \mathscr A$ obtained by differentiating the
torus $\mathbb{T}^2$  associated with the $\Z^2$-grading.
Another basic ingredient of the differential calculus is the canonical trace
\begin{align*}
\varphi_0 : \mathscr A \to \mathbb{C} :
\sum_{ (n,m) \in \Z^2 } a_{ (n,m)} U^n V^m  \mapsto  a_{ (0,0)}
\end{align*}
that simply takes the coefficient of the constant term.
We shall consider variational problems on the space of self-adjoint elements
$ \mathfrak S = \mathfrak S ( \mathscr A)$
viewed as the tangent space of a conformal class
of metrics on $ \mathscr A$, cf. \cite[\S4.2]{MR3194491}.
For a fixed self-adjoint $h$, denote by $\bar \delta_a$
the variation of $h$ along some other self-adjoint $a = a^* \in \mathscr A$:
\begin{align*}
h \mapsto h + \varepsilon a, \, \,  \, \,
\bar \delta_a \defeq
\frac{d}{d \varepsilon} \Big |_{ \varepsilon =0} .
\end{align*}
Let $F(h)$ be a functional in  $h$,  i.e., a function
$F : \mathfrak S  \to \mathbb{C}$ on  $ \mathfrak S$.
The gradient of $F$ at $h$ (associated with $\varphi_0$),
$\grad_h F \in \mathscr A$
is defined to be the unique element such that
\begin{align}
\label{eq:gradF-defn}
\bar \delta_a F(h) = \varphi_0 \brac{  \grad_h F  a}, \, \, \, \,
\forall a \in  \mathfrak S ( \mathscr A).
\end{align}

The following first variation formula was first proved in 
\cite[Thm 4.10]{MR3194491}.
The underlying co-simplicial and cyclic structure of $\mathcal C$ has already 
been revealed in the version of \cite[Thm 2.15, 2.17]{Liu:2018ab}. The new
observation from \eqref{eq:CMtp-II} is role of $\facemaps_2$ in cyclic theory,
cf. \eqref{eq:last-face-del-n+1}. 
\begin{prop} \label{prop:CM-fun-relt-C}
Consider the following functional 
\begin{align*}
F(h) =
\varphi_0 \brac{
\mathscr S_h ( T )
(\nabla h) \cdot \nabla h},
\end{align*}
for some $T \in \funsp (1)$ and $\nabla = \nabla_j$, $j=1$ or  $2$,
is one of the basic derivations.
The gradient is of the general form of a second order
differential expression purely in terms of $h$:
\begin{align*}
\grad_h F =
\mathscr S_h ( K )
\brac{ \nabla^2 h } +
\mathscr S_h ( H )
\brac{ \nabla h \otimes  \nabla h } ,
\end{align*}
where the coefficients $K \in \funsp(1)$ and $H \in \funsp(2)$ satisfy
\begin{align}
    \label{eq:CMtp-I} 
K &=  - \brac{ 1 + \pmb\tau_{ (1)}  } ( T) , \\
H &= \brac{  \pmb\delta_0 + \pmb\delta_1 - \pmb\delta_2} (K) ,
    \label{eq:CMtp-II} 
\end{align}
where $ \facemaps_j = \facemaps^{(j,1)}$, $j=0,1,2$ and recall that
$\facemaps_2 = \pmb\tau_{ (2)}  \pmb\delta_0$ is the one 
in \eqref{eq:last-face-del-n+1}.
\end{prop}
\begin{proof}
By the Leibniz rule, $ \bar \delta_a F(h) = P_{ \mathrm I} + P_{ \mathrm{II}}$
consists of two parts. First, the variation on $\nabla h$.
Note that two types of differentials commute:
$ \bar \delta_a \brac{ \nabla h } = \nabla \brac{ \bar \delta_a h } = \nabla a$,
thus
\begin{align*}
P_{ \mathrm I}
&=
\varphi_0 \brac{
\mathscr S_h (T)
(\nabla a) \cdot \nabla h } +
\varphi_0 \brac{
\mathscr S_h (T)
(\nabla h) \cdot \nabla a } \\
&=
\varphi_0 \brac{
\mathscr S_h \brac{ \brac{ 1 + \pmb\tau_{ (1)}  } (T) }
(\nabla h) \cdot \nabla a } ,
\end{align*}
where we have used \cref{prop:tau-n-defn}.
We continue by applying   integration by parts:
\begin{align*}
P_{ \mathrm I} =  - \varphi_0 \brac{
\nabla \sbrac{
\mathscr S_h \brac{
\brac{ 1 + \pmb\tau_{ (1)}  } (T)
}
(\nabla h) } \cdot a  }.
\end{align*}
Let $K = - \brac{ 1 + \pmb\tau_{ (1)}  } (T)$,
\cref{prop:face-j-for-h} yields
\begin{align*}
- \nabla \sbrac{
\mathscr S_h \brac{  \brac{ 1 + \pmb\tau_1 } (T)}
(\nabla h) }
= \mathscr S_h \brac{
(\pmb\delta^{(0,1)} + \pmb\delta^{(1,1)}) (K)
}
(\nabla h \otimes  \nabla h)
+
\mathscr S_h \brac{ K }
(\nabla^2 h) .
\end{align*}
Finally, $P_{ \mathrm I} $ has been turned into the form of
the right hand side of \eqref{eq:gradF-defn}:
\begin{align}
\label{eq:PI-done}
P_{ \mathrm I}  &=
\varphi_0 \brac{
\mathscr S_h \brac{ (\pmb\delta^{(0,1)} + \pmb\delta^{(1,1)} ) (K)   }
(\nabla h \otimes  \nabla h)  \cdot  a}
+
\varphi_0 \brac{
\mathscr S_h ( K )
(\nabla^2 h) \cdot a
}
\end{align}
For $  P_{ \mathrm{II} }$ concerning only the variation on  the rearrangement
operator
$ \mathscr S_h (T)$,
\begin{align*}
P_{ \mathrm{II} }
&=
\varphi_0\brac{
\bar \delta_a \sbrac{
\mathscr S_h (T)
}
(\nabla h) \cdot \nabla h} ,
\end{align*}
we need again \cref{prop:face-j-for-h}:
\begin{align*}
\bar \delta_a \sbrac{
\mathscr S_h (T)
} (\nabla h)
=
\mathscr S_h \brac{  \facemaps^{(0,1)} (T) }
(a \otimes  \nabla h)
+
\mathscr S_h \brac{  \facemaps^{(1,1)} (T) }
(\nabla h \otimes a ) ,
\end{align*}
and then move $a$ to the very right (as required in  \eqref{eq:gradF-defn})
using the cyclic operator  in \cref{prop:tau-n-defn}:
\begin{align*}
P_{ \mathrm{II} }
&=
\varphi_0 \brac{
\brac{
\mathscr S_h \brac{
\pmb\tau_{ (2)}  \facemaps^{(0,1)} (T) +
\pmb\tau_{ (2)} ^2 \facemaps^{(1,1)} (T) }
}
( \nabla h \otimes  \nabla h) \cdot a
}.
\end{align*}
The first term above is the
last face operator $\facemaps^{(2,1)} =  \pmb\tau_{ (2)}  \facemaps^{(0,1)} $
(see \eqref{eq:last-face-del-n+1})
and second term can also be modified using \eqref{eq:tau-vs-f-d-thm}:
\begin{align*}
\pmb\tau_{ (2)}^2 \facemaps^{(1,1)} =
\pmb\tau_{ (2)}   \facemaps^{(0,1)} \pmb\tau_{ (1)}
=  \facemaps_2 \pmb\tau_{ (1)} ,
\end{align*}
thus the final form of $ P_{ \mathrm{II}} $ is given by
\begin{align}
\label{eq:PII-done}
P_{ \mathrm{II} }
&=
\varphi_0 \brac{
\mathscr S_h \brac{
\facemaps_2 (1+ \pmb\tau_{ (1)} ) (T)
}
( \nabla h \otimes  \nabla h) \cdot a
}.
\end{align}

The desired relations of $K$ and  $H$ follow immediately after adding up
$P_{ \mathrm I} $ and $P_{ \mathrm{II}} $
in \eqref{eq:PI-done} and \eqref{eq:PII-done}.

\end{proof}

In \cite[Thoerem 4.10]{MR3194491}, 
the rearrangement operators are defined as the functional calculus
in terms of the modular operator $\mathbf x = [ \cdot , h]$,
cf. \cite[\S1.5]{MR3194491}.
The gap between the two versions is  a simple  process of
changing variables, that is, from multiplication operators
$\set{h^{(0)}, \ldots , h^{(n)}}$ to
modular derivations
$\set{h^{(0)}, \mathbf x^{(1)} , \ldots, \mathbf x^{(n)}}$.
Following the notations in \cite[\S 3.2]{leschdivideddifference}, we set: 
\begin{align}
    \label{eq:x-to-h} 
\mathbf x^{ (j) } = h^{(j)}  - h^{(j-1)}, \, \,  j=1, \ldots n.
\end{align}
In other words,
$\mathbf x^{(j)} = - [ \cdot  , h]^{(j)}: \mathscr A^{\otimes n} \to \mathscr A$
is the commutator derivation acting on the $j$-factor of a elementary tensor.
 Accordingly, one defines the Schwartz functional calculus by the substitution:
\begin{align*}
\mathscr S_{ \mathbf x} ( f ) \defeq
f (\mathbf x^{(1)}, \ldots, \mathbf x^{(n)})  =
f \brac{ h^{(1)} - h^{(0)} , \ldots , h^{(n)} -  h^{(n-1)}}
\in L \brac{ \mathscr A^{\otimes n} , \mathscr A}.
\end{align*}
Now let $ \tilde K$ be the one-variable function such that
$ \tilde K (\mathbf x)
= \tilde K (h^{(1)} - h^{(0)}) = K (h^{(0)} , h^{(1)})$,
then
\begin{align*}
\pmb\delta_0 (K) \brac{  h^{(0)} , h^{(1)} , h^{(2)}}
&=
K (\bullet , h^{(2)}) [h^{(0)} , h^{(1)}] =
\frac{ K \brac{ h^{(0)} , h^{(2)} } - K \brac{ h^{(1)} , h^{(2)} } }{
h^{(0)} - h^{(1)} } \\
&=
\frac{ \tilde K (\mathbf x^{(1)} + \mathbf x^{(2)}) -
\tilde K ( \mathbf x^{(1)})
}{- \mathbf x^{(1)}}.
\end{align*}
Similarly,
\begin{align*}
\pmb\delta_1 (K) \brac{  h^{(0)} , h^{(1)} , h^{(2)}} &=
K ( h^{(0)} ,\bullet  ) [h^{(1)} , h^{(2)}] =
\frac{
\tilde K ( \mathbf x^{(1)})   -
\tilde K (  \mathbf x^{(1)} + \mathbf x^{(2)})
}{ - \mathbf x^{(2)}}
\end{align*}
and
\begin{align*}
\pmb\delta_2 (K) \brac{  h^{(0)} , h^{(1)} , h^{(2)}} &=
\brac{ \pmb\tau_2\pmb\delta_0 } (K)
\brac{  h^{(0)} , h^{(1)} , h^{(2)}} =
K (  \bullet, h^{(1)}  ) [h^{(2)} , h^{(0)}]  \\
&=
\frac{
\tilde K (- \mathbf x^{(2)}) -
\tilde K (\mathbf x^{(1)})
}{
\mathbf x^{(2)} + \mathbf x^{(1)}
} .
\end{align*}
Provided that $ \tilde K$ is an even function, i.e.
$ \tilde K (- \mathbf x^{(2)}) = \tilde K ( \mathbf x^{(2)})$,
we  see that, upto a minus sign,
$( \pmb\delta_0 + \pmb\delta_1 - \pmb\delta_2) (K)$ is exactly the right
hand side of \cite[Eq. (4.42)]{MR3194491}.

\subsection{Functions Arising in Pseudo-differential Calculus}

\label{sec:fn-psdo} 

We set 
$  \omega_\lambda ( x_0 )= \brac{ x_0 - \lambda  }^{-1} \in  \funsp (0)$, 
where $\lambda$ is a complex parameter. 
For each $\alpha = (\alpha_0 , \ldots , \alpha_n) \in \Z_{ \ge 0}^{ n+1 } $,
consider 
\begin{align}
\label{eq:omega-alp}
\omega_\alpha \defeq
\omega_\lambda^{\alpha_0} \otimes \cdots  \otimes  \omega_\lambda^{ \alpha_n}
\in  \funsp (n),
\end{align}
that is
$ \omega_\alpha  \brac{ x_0 , \ldots x_n } =
\omega_\lambda(x_0)^{ \alpha_0 }
\cdots  \omega_\lambda^{ \alpha_n} ( x_n ) $.
Obviously, the cyclic operator $\cyc_{ (n)} $  in \eqref{eq:cyc_n-sp-var}
leads to cyclic permutation on  the index $\alpha$:
\begin{align*}
\omega_{ (\alpha_0 , \ldots , \alpha_n)} \xrightarrow{\cyc_{ (n)} }
\omega_{ (\alpha_n , \alpha_0 , \ldots , \alpha_{ n-1} )} .
\end{align*}

Let us compute the divided difference of $\omega_\lambda$:
\begin{align*}
\omega_\lambda [x_0 , x_1 ] =
\frac{ (x_0 - \lambda)^{-1} -  (x_1 - \lambda)^{-1} }{ x_0 - x_1} =
- \omega_\lambda (x_0) \omega_\lambda (x_1) ,
\end{align*}
which, in terms of the coproduct notation in \eqref{eq:divdif-coprd-notn},
can be written as
$ \ddif ( \omega_\lambda) =  \omega_\lambda \otimes   \omega_\lambda
= - \omega_{ (1,1) } $.
It follows from \eqref{eq:fac_n-sp-var} that  both
$ \facemaps^{(0,1)} \brac{ \omega_{ (1,1)}  }   $ and
$ \facemaps^{(1,1)} \brac{ \omega_{ (1,1)}  }   $
are equal to $ \omega_{ (1,1,1)}$, for instance:
\begin{align*}
\facemaps^{(0,1)} \brac{ \omega_{ (1,1)}  } =
\brac{ \ddif \otimes 1 } \brac{ - \omega_\lambda \otimes  \omega_\lambda }
=
- ( - \omega_\lambda \otimes \omega_\lambda) \otimes \omega_\lambda
= \omega_{ (1,1,1)}  .
\end{align*}
In a similar way, we can obtain all $ \omega_{ (1 , \ldots ,1)} \in  \funsp(n) $
by iterating the face operators on $\omega_\lambda \in  \funsp(0)$:
\begin{equation}
    \label{eq:fac-omg-lam} 
    \begin{tikzcd}
\omega_\lambda \arrow{r}{\facemaps^{(0,0)}}
& - \omega_\lambda \otimes \omega_\lambda
\arrow[bend left=20]{r}{\facemaps^{(0,1)}}
\arrow[bend right=20, ']{r}{\facemaps^{(1,1)}}
& \omega_\lambda \otimes \omega_\lambda \otimes \omega_\lambda
\arrow[bend left=20]{r}{\facemaps^{(0,1)}}
\arrow{r}{\facemaps^{(1,1)}}
\arrow[bend right=20 , ']{r}{\facemaps^{(2,1)}}
& - \omega_\lambda \otimes \omega_\lambda \otimes \omega_\lambda
\otimes \omega_\lambda
\quad \cdots .
\end{tikzcd}
\end{equation}
For the action of the degeneracy operators, \eqref{eq:deg_n-sp-var} yields
\begin{align}
\label{eq:degj-omega-l}
\degenmaps^{(j,n)} \brac{ \omega_\alpha } =
\omega_\lambda^{ \alpha_0 } \otimes \cdots \otimes
\omega_\lambda^{ \alpha_{ j-1}  } \otimes
\omega_\lambda^{ (\alpha_j + \alpha_{ j+1})  } \otimes
\omega_\lambda^{ \alpha_{ j+1}  } \otimes \cdots \otimes
\omega_\lambda^{ \alpha_{ n }  }
= \omega_{ \alpha'} ,
\end{align}
where $\alpha' = (\alpha_0 , \ldots, \alpha_j + \alpha_{ j+1 }, \ldots, n)$.
Now we can reach any $\omega_\alpha$,  $\alpha \in  \Z_{ \ge 0}^{n+1} $,
from $\omega_\lambda$ by iterating the face operators to produce
$\omega_{ (1, \ldots ,1)} $ and then contract the indices
according to \eqref{eq:degj-omega-l},
for instance,
\begin{align*}
\degenmaps^{(0,3)} \brac{ \omega_{ 1,1,1,1}  } =  \omega_{ 2,1,1},
\, \, \, \,
\degenmaps^{(1,3)} \brac{ \omega_{ 1,1,1,1}  } =  \omega_{ 1,2,1} ,
\, \, \, \,
\degenmaps^{(1,2)} \brac{ \omega_{1,2,1}  } =  \omega_{ 1,3} .
\end{align*}

Relations in \cref{thm:cyc-cato-sim-I} can also be used to derived differential
relations.
The simplest case would be  $\partial_{ x_0} \omega_\lambda = - \omega_\lambda^2$,
which can be seen by writing
$ \partial_{ x_0} =  \degenmaps^{(0, 1)} \facemaps^{(0, 0)} $,
while the right hand side can be computed directly:
\begin{align*}
\omega_\lambda
\xrightarrow{ \facemaps^{(0, 0)} }
 - \omega_\lambda \otimes \omega_\lambda
\xrightarrow{  \degenmaps^{(0, 1)}}
-\omega_\lambda^2 .
\end{align*}
The next one: 
$ \partial_{ x_0}^2 \omega_\lambda = 2 \omega_\lambda^3 $.
As before, we first rewrite the second derivative
$\partial_{ x_0}^2 = \degenmaps_0 \facemaps_0 \degenmaps_0 \facemaps_0 $
using \eqref{eq:degj-faci-thm} and \eqref{eq:degj-degi-thm}:
\begin{align*}
\degenmaps^{(0 , 1)} \degenmaps^{(0,2)} \facemaps^{(0 , 1)} \facemaps^{(0 , 0)} 
= &\, \,
\degenmaps^{(0 , 1)} \degenmaps^{(1 , 2)} \facemaps^{(0 , 1)} \facemaps^{(0 , 0)} 
=
\degenmaps^{(0,1)}
\brac{
\facemaps^{(0 , 0)} \degenmaps^{(1 , 1)} - \degenmaps^{(0 , 2)} \facemaps^{(1 , 1)}
}
\facemaps^{(0,0)}  \\
    = &\, \,
\degenmaps^{(0 , 1)} \facemaps^{(0,0)} \degenmaps^{(0 , 1)} \facemaps^{(0 , 0)} -
\degenmaps^{(0 , 1)} \degenmaps^{(0 , 2)} \facemaps^{(1 , 1)} \facemaps^{(0 , 0)} .
\end{align*}
As we have seen in \eqref{eq:fac-omg-lam} and \eqref{eq:degj-omega-l}:
\begin{align*}
\omega_\lambda^3 =
\degenmaps^{(0 , 0)} \degenmaps^{(0 , 2)}
\facemaps^{(0 , 1)} \facemaps^{(0 , 0)}  (\omega_\lambda ) =
\degenmaps^{(0 , 0)} \degenmaps^{(0 , 2)}
\facemaps^{(1 , 1)} \facemaps^{(0 , 0)}  (\omega_\lambda ) .
\end{align*}
Therefore the sum 
$ \brac{ 
\degenmaps^{(0 , 0)} \degenmaps^{(0 , 2)}
\facemaps^{(0 , 1)} \facemaps^{(0 , 0)}
+
\degenmaps^{(0 , 0)} \degenmaps^{(0 , 2)}
\facemaps^{(1 , 1)} \facemaps^{(0 , 0)}   
} (\omega_\lambda )$
indeed recovers the coefficient $2$ in
$ \partial_{ x_0}^2 \omega_\lambda = 2 \omega_\lambda^3 $.
One can push the computation above to derive the following differential relations
\begin{align*}
\omega_{ (\alpha_0 , \ldots , \alpha_n)} =
\brac{ \prod_{ j=0}^n \frac{ (-1)^{\alpha_j-1}}{ (\alpha_j -1)!}
\partial_{ x_j}^{\alpha_j - 1}
}  \omega_{ (1, \ldots ,1)} .
\end{align*}

The family $\set{ \omega_\alpha}$ is intricately connected to the hypergeometric one
$\set{ H_\alpha}$ introduced in the rearrangement lemma in 
\cite{Liu:2018aa}, which a crucial technical tool in the  pseudo-differential
calculus approach of deciphering heat asymptotic. 
A more general version of the lemma  designed for going beyond
conformal geometry has been achieved in \cite{Liu2020General-Rearran}.
The calculation on $\set{ \omega_\alpha}$  hints that one
might be able to derive the recursive and differential relations of
$\set{ H_\alpha}$ in the hypergeometric literature from the action of the
category $ \mathcal C$. 
What has been used is the observation that the cyclic $ \cyc_{ (n)} $
operators acts as cyclic permutation on the index $\alpha$,
which is the new input to obtain 
computer-aid free verification of the Connes-Moscovici type functional relation,
cf. \cite[\S~5]{Liu:2018ab}.
We postpone the fully investigation to future publications and remark that 
the question serve greatly to the interest of finding simplification of the 
$a_4$-term computation \cite{2016arXiv161109815C}.

\appendix

\section{ The Cyclic Category $\Delta C$ }
\label{sec:cyc-DelC}

Connes's cyclic category  $\Delta C$ is a mixture of the simplicial category
$\Delta$ and the cyclic groups.
We collect related notations and preliminary results from
\cite{Loday:1992vm} for the reader's convinience.

\begin{defn}
\label{defn:DelCyc}
The cyclic category $\Delta C$ consists of objects $[n]$,  $n \in \Z_{ \ge 0}$,
and morphisms generated by faces $\delta_j^{(n-1)}: [n-1] \to [n]$, degeneracies
$\sigma_j^{(n+1)}: [n+1] \to [n]$, $j=0, \ldots , n$, and cyclic operators
$\tau_{ (n)}  : [n] \to [n] $, subject to the following relations
\footnote{
The superscripts which indicate the domain have been suppressed accordingly,
such as $\sigma_j^{(n)} \mapsto  \sigma_j$. }:
\begin{itemize}
\item simplicial and co-simplicial relations:
\begin{align}
\label{eq:simp-rels-I}
&    \delta_j \delta_i = \delta_i \delta_{ j-1} ,\, \,
\text{ for $i<j$, }\\
&    \sigma_j \sigma_i = \sigma_i \sigma_{ j +1}  ,\, \,
\text{ for $i \le j$; }
\nonumber
\end{align}
\item their compatibility:
\begin{align}
\label{eq:simp-rels-II}
\sigma_j \delta_i =
\begin{cases}
\delta_i \sigma_{ j-1}  &  i<j \\
\mathbf 1_{ [n]}        & i=j, i=j+1 \\
\delta_{ i-1} \sigma_j  & i> j+1
\end{cases}
\end{align}

\item for cyclic operators: $\tau_{ (n)} ^{n+1} = 1$.

\item compatibility of the cyclic and the simplicial structures:
\begin{align}
\label{eq:simp-cyc-rels}
\tau_{ (n)}  \delta_i = \delta_{ i-1} \tau_{ (n-1) }, \, \,
\tau_{ (n)}  \sigma_i =   \sigma_{ i-1} \tau_{ (n+1) }, \, \,
\end{align}
where $i=1, \ldots, n$.
\end{itemize}
\end{defn}
\begin{rem}
The compatibility relations \cref{eq:simp-cyc-rels} is equivalent to
\begin{align}
\label{eq:simp+cyc-rels-V-2-1}
\delta_n = \tau_{ (n)} ^{n+1} \delta_{ n-1} =
\tau_{ (n)} ^n \delta_{ n-1} \tau_{ n-1} =
\cdots = \tau_{ (n)}  \delta_0 \tau_{ (n-1) } ^n =
\tau_{ (n)}  \delta_0
\end{align}
and
\begin{align}
\label{eq:simp+cyc-rels-V-2-2}
\tau_{ (n)}  \sigma_0 =
\tau_{ (n)}  \sigma_0 \tau_{ (n+1)}^{n+2} =
\tau_{ (n)}^2 \sigma_1 \tau_{ (n+1)}^{n+1} = \cdots =
\tau_{ (n)}^{n+1} \sigma_n \tau_{ (n+1)}^{2} =
\sigma_n \tau_{ (n+1)}^{2}  .
\end{align}
\end{rem}

The simplicial category $\Delta$ is the  subcategory of $\Delta C$ having
the same objects $[n]$,  $n=0,1,2 \ldots $,
while the morphisms are generated  by the faces and degeneracies.
The following realization is related to the discussion in \cref{sec:vcal-reops}.
The objects $[n]$ are ordered set  $\set{0<1 < \cdots < n}$ of $ n+1$ points.
The morphisms consists of non-decreasing functions $f: [n] \to [m]$, meaning
$f(i) \ge f(j)$ whenever $i >  j$.
They can be classified into two categories: injective  and surjective  ones.
The non-decreasing property implies that any morphism
$\psi : [n] \to [m]$ can be represented by the ordered list
$ \set{ \psi^{-1} (\set{0}) , \ldots , \psi^{-1} (\set{m}) }$
the pre-images \footnote{
The notation $\psi^{-1}$ should not be confused with the inverse of $\psi$.
}
of points in $[m]$.
When $\psi$ is injective, the cardinality
$ \# \psi^{-1} \brac{ \set{ l} } $ is either $0$ or  $1$.
Thus an injective morphism is determined by the collection of missing points
\begin{align}
\label{eq:mi-missing-pts}
\mathrm{ms} (\psi) =
\set{ l \in  [m] :  \# \psi^{-1} \brac{ \set{ l} } =0}.
\end{align}
Similarly, surjectivity means that
$ \# \psi^{-1} \brac{ \set{ l} } \ge  1$
and the collection of pre-images $\psi^{-1} \brac{ \set{ l} }$
at the stationary points
\begin{align}
\label{eq:st-stationary-pts}
\mathrm{st}( \psi) =\set{
l \in [m]    :   \# \psi^{-1} \brac{ \set{ l} } > 1
}
\end{align}
is sufficient to capture the morphism $\psi$.
It follows that the identities $\mathbf 1_{ [n]}$ are the only isomorphisms
(i.e. non-decreasing bijections) in $\Delta$.

As for the generators listed in \cref{defn:DelCyc},
the face map $ \delta_j^{(n -1)}: [n-1] \to [n]$
is the injective function missing $  j \in  [n]$
and the degeneracy
$ \sigma_j^{( n +1)} : [n+1] \to [n]$ is the surjective function
hitting  $ j \in  [n] $ twice: $ \sigma_j^{( n +1)} (j) =
\sigma_j^{( n +1)} (j +1)  =j$.

\begin{prop}
\label{prop:decom-phi-monotone}
For any morphism $ \phi: [n] \to [m]$ in the simplicial category $\Delta$,
there is a unique decomposition
\begin{align}
\label{eq:decom-phi-monotone}
\phi = \delta_{ i_1}  \cdots \delta_{ i_r}
\sigma_{ j_1} \cdots \sigma_{ j_s},
\end{align}
such that $i_1 >  i_2 >  \cdots > i_r$ and
$j_1 <  j_2  <  \cdots  < j_s$ with $m = n -s +r$. Convention: if the index
set is empty, then $\phi$ is the identity.
\end{prop}
\begin{proof}
See \cite[Appendix B]{Loday:1992vm} and the references therein.
\end{proof}

To complete the corresponding realization of $\Delta C$
(from $\Delta$ described above),
we choose the following cyclic permutations, derived from the cycle
$(01 \ldots n) \in S_{ n+1} $,
as the generators of cyclic groups
\begin{align*}
\tau_{(n)} : [n] \to [n]: 0\mapsto 1 , 1 \mapsto  2 , \ldots , n\mapsto  0.
\end{align*}

The cyclic groups and the simplicial category $\Delta$ are combined in a way that
any morphism in $\Delta C$ can by written uniquely as a composition of a
morphism in $\Delta$ and an element of the cyclic groups. Moreover,
from  \cref{eq:simp+cyc-rels-V-2-1,eq:simp+cyc-rels-V-2-2}, we see that with
the cyclic operators $\tau_{ (n)} $,
one only needs another face and degeneracy maps,
say, $\sigma_0^{(n)}$ and $\delta_0^{(n)}$, to recover the rest.

Another notable feature  of the cyclic category $\Delta C$ (not true for the
simplicial category) is the self-duality $\Delta C \cong \Delta C^{\mathrm{op}}$.
By taking $ \delta_i^* = d_i$,  $\sigma_j^* = s_j$ and  $\tau_n^* = t_n$,
we have the following presentation of $\Delta C^{\mathrm{op}}$.
\begin{defn}
\label{defn:DelCyc-OP}
The opposite category  $\Delta C^{\mathrm{op}}$ has a presentation
given by generators  $t_{ (n)}: [n] \to [n]$ and
\begin{align*}
d_i^{(n)} : [n] \to [n -1], \, \, \, \,
s_i^{(n)}: [n] \to [n + 1], \, \, \, \, i=0, \ldots , n.
\end{align*}
The corresponding relations are
\begin{itemize}
\item simplicial and co-simplicial:
\begin{align}
\label{eq:coCyc-didj}
& d_i d_j = d_{ j-1} d_i , \, \, i < j  \\
& s_i s_j = s_{ j+1} s_i , \, \, i\le j
\label{eq:coCyc-sisj}
\end{align}

\item compatibility:
\begin{align}
d_i s_j =
\label{eq:coCyc-disi}
\begin{cases}
s_{ j -1 } d_i & \text{ for $i < j$} \\
\mathrm{id}_{ (n)} & \text{ for $i=j$,  $i= j+1$} \\
s_{ j  } d_{  i-1}  & \text{ for $i > j + 1$}
\end{cases}
\end{align}
and
\begin{align}
\label{eq:coCyc-ditn}
d_i t_{ (n)} = t_{ (n-1)} d_{ i -1}, \, \,
\text{ for $ 1\le  i \le n$} \, \,
\text{ and} \, \, d_0 t_{ (n)} = d_n  ,
\\
s_i t_{ (n)} = t_{ (n+1)} s_{ i -1}
\text{ for $ 1\le  i \le n$} \, \,
\text{ and} \, \, s_0 t_{ (n)} = t_{ (n + 1)}^2 s_n    .
\label{eq:coCyc-sitn}
\end{align}

\item cyclic: $ t_{ (n)}^{n+1}  = \mathrm{id}_{ (n)} $.
\end{itemize}

\end{defn}

We will also make use of the  duality functor from
$ \Delta C^{\mathrm{op}}$ to $\Delta C$ described in
\cite[\S6.1.11]{Loday:1992vm}.
The construction starts with adding the following extra degeneracies
to the generators of $\Delta C$:
\begin{align}
\label{eq:extr-deg-cyc}
\sigma^{(n+1)}_{ n+1}  \defeq \sigma_0 t_{ n+1 } ^{-1} : [ n+1] \to [n] .
\end{align}
Notice that the relation $ t_n \sigma_i = \sigma_{ i-1} t_{ n+1}  $
extends to $i=n+1$.

The duality functor sends $[n] \to [n]$, on morphisms,
it sends $ \tau_n \to (t_n)^{-1} : [n ] \to [n] $ and
\begin{align}
\label{eq:dul-futr-d-sig}
d_i^{(n+1)} \mapsto  \sigma_i^{(n+1)} : [n+1] \to [n] , \, \,
\text{ for $i= 0 , \ldots , n+1$,}
\end{align}
and
\begin{align}
\label{eq:dul-futr-s-del}
s_i^{(n-1)} \to \delta_{ i+1}^{(n)} : [n-1] \to [n] , \, \,
\text{ for $i= 0 , \ldots , n$.}
\end{align}
Note that $\delta_0^{(n)}$
is missed in \cref{eq:dul-futr-s-del}.

\bibliographystyle{hplain}

\bibliography{mylib1.bib}
\end{document}